\documentclass[12pt]{amsart}

\usepackage{MnSymbol}
\usepackage{graphicx}

\newtheorem{theorem}{Theorem}[section]
\newtheorem{lemma}[theorem]{Lemma}

\newtheorem{corollary}[theorem]{Corollary}

\theoremstyle{definition}

\newtheorem{example}[theorem]{Example}

\theoremstyle{remark}
\newtheorem{remark}[theorem]{Remark}

\newcommand{\rr}{{\mathbb R}}
\newcommand{\ganz}{{\mathbb Z}}
\newcommand{\nat}{{\mathbb N}}
\newcommand{\complex}{{\mathbb C}}
\newcommand{\AR}{\operatorname{AR(1)}}
\newcommand{\eps}{{\varepsilon}}
\newcommand{\Pm}{{\mathbb P}}
\newcommand{\Qm}{{\mathbb Q}}

\numberwithin{equation}{section}

\makeatletter
\@namedef{subjclassname@2020}{\textup{2020} Mathematics Subject Classification}
\makeatother

\title[Max-semistable extremal behavior of AR(1)-processes]{Max-semistable extremal behavior of AR(1)-processes connected with Bernoulli convolutions}

\author[P. Kern]{Peter Kern}
\address{Peter Kern, Mathematical Institute, Heinrich Heine University D\"usseldorf, Universitätsstra{\ss}e 1, 40225 D\"usseldorf, Germany} \email{kern@hhu.de}

\author[A.E. Sterk]{Alef E. Sterk}
\address{Alef E. Sterk, Bernoulli Institute for Mathematics, Computer Science and Artificial Intelligence, University of Groningen, PO Box 407, 9700 AK Groningen, The Netherlands} \email{a.e.sterk@rug.nl}

\date{\today}

\begin{document}

\allowdisplaybreaks

\begin{abstract}
We consider simple autoregressive processes of type $\AR$, whose stationary distribution is supported on a subset of the unit interval and is an affine transformation of a Bernoulli convolution. A new structural representation of the stationary distribution as a product of a power function with a log-periodic function near the origin is given, which gives structural insight to the stationary distribution on the whole unit interval by using a characteristic functional equation. This enables to prove that the stationary distribution of the $\AR$-process belongs to the domain of geometric partial attraction of a max-semistable law. We further prove uniform convergence of the distribution function of normalized maxima of the $\AR$-process to a certain power of the max-semistable law in the spirit of a merge theorem and point out connections to deterministic and random dynamical systems.
\end{abstract}


\maketitle

\section{Introduction}

We consider autoregressive processes $(X_{k})_{k\in\nat_{0}}$ with parameter $\beta\in(0,1)$ of the form
\begin{equation}\label{ar1}
X_{k+1}=\beta X_{k}+\eps_{k+1},\quad k\in\nat_{0},
\end{equation}
where $X_{0}$, $(\eps_{k})_{k\in\nat}$ are independent random variables and for all $k\in\nat$
\begin{equation*}
\Pm(\eps_{k}=0)=p\quad , \quad \Pm(\eps_{k}=1-\beta)=q
\end{equation*}
for some $p\in(0,1)$ and $q=1-p$. The paper \cite{Sterk} investigates the extremal behavior for $\beta\in(0,\frac12]$ and one of our aims is to study to what extent the results of \cite{Sterk} generalize to parameters $\beta\in(0,1)$.

It is well-known that the $\AR$ process \eqref{ar1} admits a unique strictly stationary solution for all $\beta\in(0,1)$ which can be represented as 
\begin{equation}\label{statsol}
X_{k}=\sum_{j=0}^{\infty}\beta^{j}\eps_{k-j},\quad k\in\nat_{0},
\end{equation}
when extending the i.i.d.\ noise sequence to $(\eps_{k})_{k\in\ganz}$; see \cite{BroDav} or \cite{GolMal} for details. 
Clearly, the support of the stationary distribution is contained in $[0,1]$. For $\beta\in(0,\frac12)$ the support is a Cantor-like set as shown in \cite{Sterk} and coincides with the classical ternary Cantor set for $\beta=\frac13$ as also shown in Example 2.5.11 of \cite{BDM}.

For $j\in\nat_{0}$ we define $Z_{j}:=2(1-\beta)^{-1}\eps_{1-j}-1$, then
$$\Pm(Z_{j}=-1)=p\quad , \quad \Pm(Z_{j}=1)=q$$
and by \eqref{statsol} we obtain
\begin{equation}\label{berconv}
\sum_{j=0}^{\infty}\beta^{j}Z_{j}=\sum_{j=0}^{\infty}\beta^{j}\left(\frac{2\eps_{1-j}}{1-\beta}-1\right)=\frac1{1-\beta}(2X_{1}-1).
\end{equation}
The left-hand side of \eqref{berconv} is called a Bernoulli convolution and is an affine transformation of the stationary solution \eqref{statsol} of the $\AR$ process. There exists an extensive literature on Bernoulli convolutions, for an overview we refer to the monograph \cite{BSS}, the survey articles \cite{PSS, Varju} and Section 2.5 in \cite{DiaFre}. It is known by \cite{JesWin} that the distribution of a Bernoulli convolution \eqref{berconv} is either continuous singular or absolutely continuous with respect to Lebesgue measure depending on the parameters, i.e.\ there is no discrete part and the mixing case cannot appear. The support of a Bernoulli convolution is contained in $[-(1-\beta)^{-1},(1-\beta)^{-1}]$ and in the symmetric case $p=q=\frac12$ the following is known:
\begin{itemize}
\item For $\beta\in(0,\frac12)$ the distribution is singular, for $\beta=\frac12$ it is uniformly distributed and thus absolutely continuous. If $\beta\in(\frac12,1)$
is the reciprocal of a Pisot number, then the distribution is singular. These results date back to \cite{KerWin,Erdoes}.
\item A celebrated result in \cite{Sol1} with a simplified proof in \cite{PerSol1} shows that for Lebesgue almost every $\beta\in(\frac12,1)$ the distribution of a Bernoulli convolution is absolutely continuous.
\end{itemize}
A Pisot number $1/\beta\in(1,2)$ is an algebraic integer, i.e.\ the root of a polynomial with integer coefficients and leading coefficient $1$, such that all other roots of the polynomial belong to the open unit disc in $\complex$. 
The smallest limit point of Pisot numbers is the golden ratio $\frac1{\beta}=\frac{1+\sqrt{5}}{2}\approx 1.618033$ belonging to $\beta\approx 0.618034$ which is a root of the polynomial $x^2-x-1$ and the only quadratic Pisot number in $(1,2)$. Examples of algebraic integers in $(1,2]$ that are not Pisot numbers are $\frac1{\beta}=2^{1/k}$ for $k\in\mathbb N$ (root of $x^{k}-2$), the roots of $x^3-2x-2$, or $x^{n+m}-x^n-2$ for $n,m\in\nat$ with $\max\{n,m\}\geq2$. These are examples of Garsia numbers, i.e.\ algebraic integers in $(1,2)$ whose minimal polynomial have constant term $\pm2$ and all roots lie outside the closed unit disc in $\mathbb C$. Garsia numbers $\frac1{\beta}$ always lead to absolutely continuous Bernoulli convolutions with bounded or even continuous Lebesgue density as shown in \cite{Gar} and \cite{Yu}.
It is still an open question, whether there exist exceptional $\beta\in(\frac12,1)$ other than reciprocals of Pisot numbers such that the distribution of the corresponding Bernoulli convolution is singular. Some illustrative numerical examples are given in \cite{Sol2}. 

In the asymmetric case $p\not=q$ similar results hold. In \cite{PerSol2} it is shown that the distribution of a Bernoulli convolution is singular for all $\beta\in(0,p^{p}(1-p)^{1-p})$ and absolutely continuous for almost every $\beta\in(p^{p}(1-p)^{1-p},1)$, where the reciprocals of Pisot numbers in $(1,p^{-p}q^{-q})$ still provide examples of exceptional parameters with singular distribution; see also \cite{PerSch} for details. Note that $f(p):=p^{p}(1-p)^{1-p}$ attains its minimal value for $p=\frac12$ with $f(\frac12)=\frac12$ and $f(p)\to1$ as $p\downarrow0$ or $p\uparrow 1$. Only recently specific examples of absolutely continuous Bernoulli convolutions in the asymmetric case were given in \cite{Varju2}, in particular all Garsia numbers sufficiently close to 1 have this property.

In view of \eqref{berconv} all these results on Bernoulli convolutions directly transfer to the corresponding stationary distribution \eqref{statsol} of the $\AR$. In our study we cannot contribute to the open question on characterizing the parameters that lead to absolute continuity or singularity, but rather give a structural result for the stationary distribution function in Section 2 that enables us to prove that for i.i.d.\ sequences it belongs to the domain of geometric partial attraction of a max-semistable law also for $\beta\in(\frac12,1)$, which has already been shown for the case $\beta\in(0,\frac12]$ in \cite{Sterk}.

A distribution function $F$ belongs to the domain of geometric partial attraction if for an i.i.d.\ sequence $(X_n)_{n\in\nat}$ distributed as $F$ and the partial maxima $M_n^\circ=\max\{X_1,\ldots,X_n\}$ we have as $n\to\infty$
\begin{equation}\label{maxdogpa}
\Pm\big(a_n(M_{k(n)}^\circ-b_n)\leq x\big)=\big(F(a_n^{-1}x+b_n)\big)^{k(n)}\to H(x)
\end{equation}
for all continuity points of a non-degenerate distribution function $H$, a subsequence $(k(n))_{n\in\nat}$ of positive integers such that  $\frac{k(n+1)}{k(n)}\to c\geq1$, and certain normalizations $a_n>0$, $b_n\in\rr$. The limit distribution function $H$ is called max-semistable and the normalizations can be chosen in such a way that $H$ is of one of the three forms
\begin{align}
H(x) & =\exp\big(-\exp(-x)\,\nu(x)\big),\qquad x\in\rr,\label{Gumbeltype}\\
H(x) & =\exp\big(-x^{-\alpha}\,\nu(\log x)\big),\qquad x>0,\\
H(x) & =\exp\big(-(-x)^{\alpha}\,\nu(\log(-x))\big),\qquad x<0,\label{Weibulltype}
\end{align}
where $\alpha>0$ and $\nu$ is a positive, bounded and $\log(c^{1/\alpha})$-periodic function for $c>1$ with $\alpha=1$ in case of \eqref{Gumbeltype}; see \cite{Gri1}, \cite{Gri2}, or the alternative expositions in \cite{CHT}, \cite{Megyesi}, \cite{Pan} for details.

In case $c=1$ we can choose $k(n)=n$ and $\nu\equiv a$ as a positive constant function so that $F$ belongs to the well-known domain of attraction of a max-stable distribution function $H$; e.g.\ see \cite{Gal}, \cite{HF}, or \cite{Res}. In our paper we will only deal with $c>1$ and a max-semistable distribution function of Weibull-type \eqref{Weibulltype}.

Our limit result \eqref{maxdogpa} for $\beta\in(\frac12,1)$ in Section 2 only holds along the geometrically increasing subsequence $k(n)=\lfloor q^{-n}\rfloor$ and can be refined by a merge theorem \cite{Megyesi} for the full parameter range $\beta\in(0,1)$ in Section 3, which states that along the sequence of all natural numbers the distribution function of appropriately normalized maxima is uniformly close to a certain power of the limiting max-semistable distribution function. This extremal behavior for i.i.d.\ sequences is mimicked by the $\AR$ sequence with an extremal index $p$ in addition as we will show in our main result in Section 4. Since our interest lies in a general merge result, we give a direct proof based on techniques in \cite{Cher} and \cite{GMS}, rather than applying general techniques for stationary sequences in \cite{LLR} that were generalized to the max-semistable situation in \cite{TemCan}. We conclude our study showing the appearance of the stationary distribution function in certain deterministic and random dynamical systems in Section 5 for which the extremal behavior is of great importance.

\section{The stationary distribution function}

For $k\in\nat_0$ let $F_k$ denote the distribution function of $X_k$, then by \eqref{ar1} we have
\begin{equation}\label{dfrecur}\begin{split}
F_{k+1}(x) & =\Pm(X_{k+1}\leq x)=\Pm(\beta X_{k}+\eps_{k+1}\leq x)\\
&=\Pm(\beta X_{k}+\eps_{k+1}\leq x\mid \eps_{k+1}=0)\cdot p+\Pm(\beta X_{k}+\eps_{k+1}\leq x\mid \eps_{k+1}=1-\beta)\cdot q\\
&=p\cdot\Pm(X_k\leq\tfrac{x}{\beta})+q\cdot\Pm(X_k\leq\tfrac{x-1+\beta}{\beta})\\
&=p\cdot F_k(\tfrac{x}{\beta})+q\cdot F_k(\tfrac{x}{\beta}+1-\tfrac1{\beta}).
\end{split}\end{equation}
This is equation (6) in \cite{Sterk}. It follows that the distribution function $F_{\beta,p}$ of the stationary $\AR$ process \eqref{statsol} fulfills equation (7) in \cite{Sterk}:
\begin{equation}\label{dfstatsol}
F_{\beta,p}(x) =p\cdot F_{\beta,p}(\tfrac{x}{\beta})+q\cdot F_{\beta,p}(\tfrac{x}{\beta}+1-\tfrac1{\beta}).
\end{equation}
Note that $F_{\beta,p}$ is continuous on $\rr$, since the distribution of a Bernoulli convolution is either continuous singular or absolutely continuous. Moreover, the support of the stationary distribution is contained in $[0,1]$ and thus we have 
$$F_{\beta,p}(x)=\begin{cases}
1 & \quad\text{ for all }x\geq 1,\\
0 & \quad\text{ for all }x\leq 0,
\end{cases}$$
as stated in Lemma 2.2 of \cite{Sterk}. This leads to uniqueness of $F_{\beta,p}$ as a solution of \eqref{dfstatsol} that has been shown in Lemma 2.3 of \cite{Sterk} for the case $\beta\in(0,\frac12]$. Uniqueness in the case $\beta\in(\frac12,1)$ needs a different proof which works for the whole parameter region $\beta\in(0,1)$ as follows.

\begin{lemma}\label{uniqueness}
Let $F$ be a continuous distribution function {on $\rr$} with support contained in $[0,1]$ such that for fixed $p\in(0,1)$, $\beta\in(0,1)$ we have
\begin{equation}\label{functionaleq}
\displaystyle F(x) =p\cdot F(\tfrac{x}{\beta})+q\cdot F(\tfrac{x}{\beta}+1-\tfrac1{\beta})\quad\text{ for all }x\in\rr.
\end{equation} 
Then $F=F_{\beta,p}$ is the unique stationary distribution function of the $\AR$.
\end{lemma}

\begin{proof}
Assume that $F,G$ are two distribution functions fulfilling the conditions such that $\|F-G\|_\infty>0$. Since $F$ and $G$ are continuous with $F(0)=0=G(0)$ we can choose $x_0>0$ such that for all $x\in[0,x_0]$ we have $F(x)\leq\frac12\,\|F-G\|_\infty$ and $G(x)\leq\frac12\,\|F-G\|_\infty$. It follows that
\begin{equation}\label{partialunique}
|F(x)-G(x)|\leq\tfrac12\,\|F-G\|_\infty\quad\text{ for all }x\in[0,x_0].
\end{equation}
Further choose $k_0\in\nat$ such that $\beta^{k_0}\leq x_0$. Iterating \eqref{functionaleq} we can write
$$\displaystyle F(x)=p^{k_0}F(\tfrac{x}{\beta^{k_0}})+\sum_{m=1}^{k_0}p^{k_0-m}q^m\sum_{\ell=1}^{k_0 \choose m} F(h_{\ell,m}(x)),$$
where each $h_{\ell,m}$ is an affine transformation on $\rr$ not depending on $F$. Since $F(\tfrac{x}{\beta^{k_0}})=1$ and $G(\tfrac{x}{\beta^{k_0}})=1$ for all $x\geq\beta^{k_0}$, it follows that for $x\in[\beta^{k_0},1]$ we get
\begin{align*}
|F(x)-G(x)| & \leq\sum_{m=1}^{k_0}p^{k_0-m}q^m\sum_{\ell=1}^{k_0 \choose m} \big|F(h_{\ell,m}(x))-G(h_{\ell,m}(x))\big|\\
& \leq\sum_{m=1}^{k_0}{k_0 \choose m}p^{k_0-m}q^m\|F-G\|_\infty=(1-p^{k_0})\|F-G\|_\infty.
\end{align*}
Together with \eqref{partialunique} we conclude
$$\|F-G\|_\infty\leq\max\{\tfrac12,1-p^{k_0}\}\,\|F-G\|_\infty$$which contradicts our assumption $\|F-G\|_\infty>0$.
\end{proof}

Interchanging the roles of $p$ and $q=1-p$ we get as in Lemma 2.1 of \cite{Sterk}
\begin{equation}\label{dfreverse}
F_{\beta,q}(x) =1-F_{\beta,p}(1-x)\quad\text{for all }x\in\rr.
\end{equation}
In case $\beta\in(0,\frac12]$ it is proven in Proposition 2.1 of \cite{Sterk} that there exists a positive, bounded and $\log(1/\beta)$-periodic function $\nu_{\beta,p}:(-\infty,0)\to(0,\infty)$ such that 
$$F_{\beta,p}(x)=x^{\log p/\log\beta}\,\nu_{\beta,p}(\log x)\quad\text{ for all }x\in(0,1).$$
Additionally, $\nu_{\beta,p}$ is continuous, since $F_{\beta,p}$ is. This result partially extends to $\beta\in(\frac12,1)$ as follows.

\begin{lemma}\label{repstatdistr}
For $\beta\in(\frac12,1)$ and $y\in(0,\frac1{\beta}-1]$ we may write
\begin{equation}\label{repstatleft}
F_{\beta,p}(y)=y^{\log p/\log\beta}\,\nu_{\beta,p}(\log y),
\end{equation}
where $\nu_{\beta,p}:(-\infty,\log(\frac1{\beta}-1)]\to(0,\infty)$ is a continuous, $\log(1/\beta)$-periodic functon.
\end{lemma}

\begin{proof}
For $y=e^x\leq\frac1{\beta}-1$ we get by \eqref{repstatleft} using \eqref{dfstatsol}
\begin{align*}
\nu_{\beta,p}(x-\log(1/\beta)) & =\nu_{\beta,p}(x+\log\beta) = e^{-(x+\log\beta)\log p/\log\beta}\,F_{\beta,p}(e^{x+\log\beta})\\
& = (e^x)^{-\log p/\log\beta}\,\frac1{p}\,F_{\beta,p}(\beta e^x)\\
& = (e^x)^{-\log p/\log\beta}\left(F_{\beta,p}(e^x)+\frac{q}{p}\,F_{\beta,p}(e^x+1-\tfrac1{\beta})\right)\\
& = (e^x)^{-\log p/\log\beta}F_{\beta,p}(e^x)=\nu_{\beta,p}(x),
\end{align*}
since $e^x+1-\frac1{\beta}\leq0$.
\end{proof}

\begin{figure}
\includegraphics{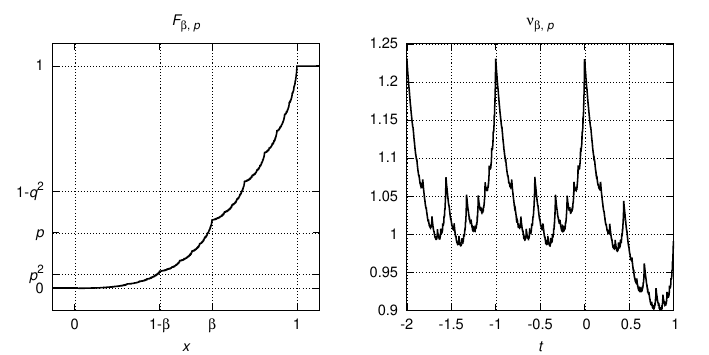}
\includegraphics{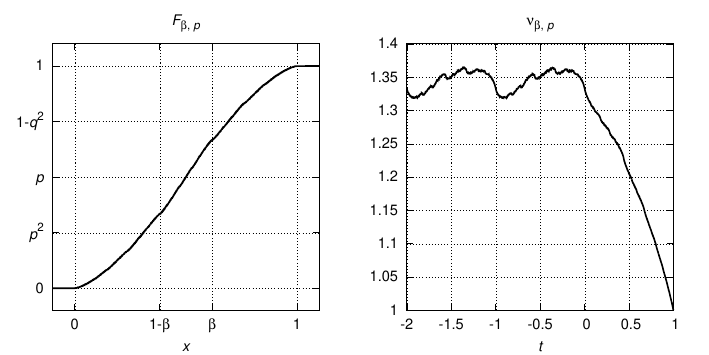}
\includegraphics{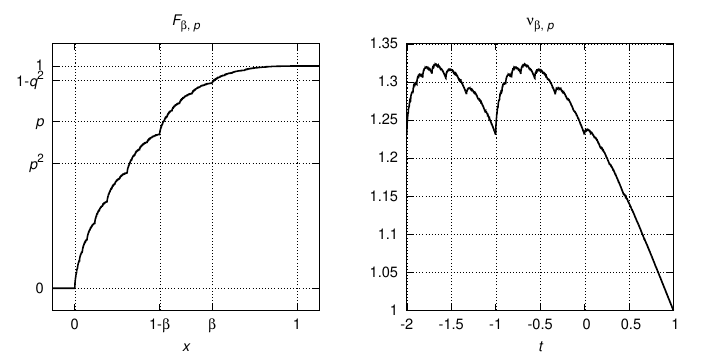}
\caption{\small Numerical approximations of the functions $F_{\beta,p}$ and $\nu_{\beta,p}$ for $\beta = 2/(1+\sqrt{5})$ and $p\in\{1/4,1/2,3/4\}$ (top to bottom). Note that the function $\nu_{\beta,p}$ is plotted against the variable $t$, where $x = (\frac{1}{\beta}-1)(\frac{1}{\beta})^t$, in order to visualize its periodicity. For $t \in [-2,0]$ two periods can be seen,  which is in agreement with Lemma \ref{repstatdistr}. However, for $t \in (0,1]$, corresponding to $x \in (\frac{1}{\beta}-1,1]$, the periodicity of $\nu_{\beta,p}$ disappears. This suggests that the representation of $F_{\beta,p}$ in Lemma \ref{repstatdistr} in general does not extend beyond the interval $(0,\frac{1}{\beta}-1]$.}
\label{fig:goldenratio}
\end{figure}

Note that the numerical approximations in Figure 1 indicate that the periodicity of $\nu_{\beta,p}$ on $(-\infty,\log(\frac1\beta-1)]$ in general cannot be extended beyond $\log(\frac1\beta-1)$. However, the representation \eqref{repstatleft} can be used to characterize $F_{\beta,p}$ in terms of $\nu_{\beta,p}$ for the full domain $y\in(0,1]$ by the following methods.\\

{\bf (A) Extension method.}
Rewriting \eqref{dfstatsol} as 
\begin{equation}\label{dfstatsolalt}
F_{\beta,p}(y) =\frac1{p}\,F_{\beta,p}(\beta y)-\frac{q}{p}\,F_{\beta,p}(y+1-\tfrac1{\beta})
\end{equation}
we may extend \eqref{repstatleft} by Lemma \ref{repstatdistr} as long as $\beta y$ and $y+1-\tfrac1{\beta}$ both belong to $(0,\frac1{\beta}-1]$, i.e.\ for $y\in(\frac1{\beta}-1,\frac1{\beta}(\frac1{\beta}-1)]$. In case $\frac1{\beta}(\frac1{\beta}-1)\geq 1$ we are done with the extension. Note that the critical case $\frac1{\beta}(\frac1{\beta}-1)=1$ leads to the golden ratio $\frac1{\beta}=\frac{1+\sqrt{5}}{2}$ which is a Pisot number. In Figure 1 we numerically illustrate that for $\frac1{\beta}=\frac{1+\sqrt{5}}{2}$ and various values of $p$ we have indeed singular behavior. If $\frac1{\beta}(\frac1{\beta}-1)>1$, we know that by Lemma \ref{repstatdistr} and \eqref{dfstatsolalt} for $y\in[1,\frac1{\beta}(\frac1{\beta}-1)]$ we have an additional functional equation
\begin{equation}\label{extright}\begin{split}
1 & = F_{\beta,p}(y) =\frac1{p}\,F_{\beta,p}(\beta y)-\frac{q}{p}\,F_{\beta,p}(y+1-\tfrac1{\beta})\\
& = y^{\log p/\log\beta}\,\nu_{\beta,p}(\log(\beta y))-\frac{q}{p}\,(y+1-\tfrac1{\beta})^{\log p/\log\beta}\,\nu_{\beta,p}(\log(y+1-\tfrac1{\beta})).
\end{split}\end{equation}
In case $\frac1{\beta}(\frac1{\beta}-1)<1$ we can again use \eqref{dfstatsolalt} to extend the representation as long as $\beta y$ and $y+1-\tfrac1{\beta}$ both belong to $(0,\frac1{\beta}(\frac1{\beta}-1)]$, i.e.\ for
$$y\in(0,\tfrac1{\beta^2}(\tfrac1{\beta}-1)]\cap(\tfrac1{\beta}-1,(\tfrac1{\beta}+1)(\tfrac1{\beta}-1) ]=(\tfrac1{\beta}-1,\tfrac1{\beta^2}(\tfrac1{\beta}-1)],$$
since $\frac1{\beta^2}\leq\frac1{\beta}+1$. Again, in case $\frac1{\beta^2}(\frac1{\beta}-1)\geq 1$ we are done with the extension and the critical case $\frac1{\beta^2}(\frac1{\beta}-1)=1$, i.e.\ $(\frac1{\beta})^3-(\frac1{\beta})^2-1=0$ is again a Pisot number. If $\frac1{\beta^2}(\frac1{\beta}-1)>1$, we know that for $y\in[1,\frac1{\beta^2}(\frac1{\beta}-1)]$ we have
\begin{equation}\label{step2}
1 = F_{\beta,p}(y) =\frac1{p}\,F_{\beta,p}(\beta y)-\frac{q}{p}\,F_{\beta,p}(y+1-\tfrac1{\beta}),
\end{equation}
but we have to distinguish between the two cases, where $\beta y$, $y+1-\tfrac1{\beta}$ belong to $(0,\frac1{\beta}-1]$ or $(\frac1{\beta}-1,\frac1{\beta}(\frac1{\beta}-1)]$ in order to apply Lemma \ref{repstatdistr} or \eqref{extright} which again leads to certain functional equations. \\
In case $\frac1{\beta^2}(\frac1{\beta}-1)<1$ we may iterate the procedure as long as $\beta y$ and $y+1-\tfrac1{\beta}$ both belong to $(0,\frac1{\beta^{2}}(\frac1{\beta}-1)]$, i.e.\ for
$$y\in(0,\tfrac1{\beta^3}(\tfrac1{\beta}-1)]\cap(\tfrac1{\beta}-1,(\tfrac1{\beta^{2}}+1)(\tfrac1{\beta}-1) ]=(\tfrac1{\beta}-1,\tfrac1{\beta^3}(\tfrac1{\beta}-1)]$$
since $\frac1{\beta^3}\leq\frac1{\beta^{2}}+1$. Thus the next critical case $\frac1{\beta^3}(\frac1{\beta}-1)=1$, i.e.\ $(\frac1{\beta})^4-(\frac1{\beta})^3-1=0$, is again a Pisot number, but successive iterations cannot always lead to critical Pisot numbers, since there is a smallest Pisot number $\frac1\beta\approx 1.324718$ as a solution of $(\frac1{\beta})^3-\frac1{\beta}-1=0$.\\
In principle, for every $\beta\in(\frac12,1)$ by a finite number of these extension steps it is possible to fully characterize $F_{\beta,p}$ on $(0,1)$ in terms of $\nu_{\beta,p}\circ\log$ on  $(0,\frac1\beta-1]$.\\

{\bf (B) Reflection method.}
Instead of property \eqref{dfstatsol} we may also use property \eqref{dfreverse} for an extension to $y\in(\frac1{\beta}-1,1]$. As long as $1-y\leq\frac1{\beta}-1$, i.e.\ $y\geq 2-\frac1{\beta}$ by \eqref{dfreverse} we have
\begin{equation}\label{extreverse}
F_{\beta,p}(y) = 1-F_{\beta,q}(1-y)=1-(1-y)^{\log q/\log\beta}\,\nu_{\beta,q}(\log(1-y)).
\end{equation}
In case $2-\frac1{\beta}\leq \frac1{\beta}-1$, i.e.\ $\beta\leq\frac23$ there is no gap and in the overlap $y\in[2-\frac1{\beta},\frac1{\beta}-1]$ both representations \eqref{repstatleft} and \eqref{extreverse} are valid, leading to
$$1-y^{\log p/\log\beta}\,\nu_{\beta,p}(\log y)=(1-y)^{\log q/\log\beta}\,\nu_{\beta,q}(\log(1-y)).$$
In case $\beta\geq\frac23$ the gap can be filled by the extension method (A), possibly combined with the reflection method.\\

As a conclusion we may say that the extension method indicates the appearance of Pisot numbers, but the above methods seem not to be suitable to get an explicit representation of $\nu_{\beta,p}$, respectively $F_{\beta,p}$. However, as for the case $\beta\in(0,\frac12]$ in \cite{Sterk}, Lemma \ref{repstatdistr} is sufficient to derive limit theorems for extremes, since the extremal behavior of an i.i.d.\ sequence distributed as $F_{\beta,p}$ is governed by the behavior of $F_{\beta,p}(x)$ as $x\uparrow1$ and hence, using \eqref{dfreverse}, by the behavior of $F_{\beta,q}(x)$ as $x\downarrow0$, which in turn is determined by the log-periodic function $\nu_{\beta,q}$ by Lemma \ref{repstatdistr}. To be more precise, let $(Y_{k})_{k\in\nat_{0}}$ be an i.i.d.\ sequence of random variables with common distribution function $F_{\beta,p}$ and define
$M_{n}^{\circ}=\max\{Y_{0},\ldots,Y_{n-1}\}$. Using an extension of the function $\nu_{\beta,q}$ from Lemma \ref{repstatdistr} to a periodic function on $\rr$, denoted by  $\nu_{\beta,q}^{\to}$, we get:

\begin{theorem}\label{maxsdoa}
For $\beta\in(\frac12,1)$ and $n\in\nat$ let $a_{n}=\beta^{-n}$, $b_{n}=1$ and $k(n)=\lfloor q^{-n}\rfloor$, then for all $x<0$ we have as $n\to\infty$
$$\Pm\left(a_{n}(M_{k(n)}^{\circ}-b_{n})\leq x\right)\to\exp\left(-(-x)^{\log q/\log\beta}\nu_{\beta,q}^{\to}(\log(-x))\right).$$
This shows that the stationary distribution $F_{\beta,p}$ of the $\AR$ belongs to the domain of geometric partial attraction of a max-semistable distribution with parameters $\alpha=\log q/\log\beta>0$ and $c=q^{-1}>1$.
\end{theorem}

\begin{proof}
In our particular situation it is much easier to give a direct proof than to show general conditions for being in a domain of geometric partial attraction given in \cite{Gri2}, \cite{Megyesi}, or \cite{CHT}.
Since $a_{n}^{-1}=\beta^{n}\to0$, for fixed $x<0$ choose $n$ large enough such that $-a_{n}^{-1}x\leq\frac1{\beta}-1$. Then by \eqref{dfreverse} and Lemma \ref{repstatdistr} we have
\begin{align*}
\Pm\left(a_{n}(M_{k(n)}^{\circ}-b_{n})\leq x\right) & =\Pm(Y_{0}\leq a_{n}^{-1}x+1)^{k(n)}\\
& =\left(1-F_{\beta,q}(-a_{n}^{-1}x)\right)^{k(n)}\\
& =\left(1-(-a_{n}^{-1}x)^{\log q/\log\beta}\nu_{\beta,q}(\log(-a_{n}^{-1}x))\right)^{k(n)}\\
& =\exp\left(k(n)\log\left(1-(-a_{n}^{-1}x)^{\log q/\log\beta}\nu_{\beta,q}(\log(-a_{n}^{-1}x))\right)\right)\\
& \sim\exp\left(-k(n)(-a_{n}^{-1}x)^{\log q/\log\beta}\nu_{\beta,q}(\log(-a_{n}^{-1}x))\right).
\end{align*}
Since by periodicity we have
$$\nu_{\beta,q}(\log(-a_{n}^{-1}x))=\nu_{\beta,q}(\log(-x)+n\log\beta)=\nu_{\beta,q}^{\to}(\log(-x))$$
and
$k(n)(a_{n}^{-1})^{\log q/\log\beta}=\lfloor q^{-n}\rfloor(\beta^{n})^{\log q/\log\beta}=\lfloor q^{-n}\rfloor q^{n}\to1,$
the assertion follows.
\end{proof}

\begin{remark}\label{rem}
Note that the normalizing sequence fulfills $a_{n}\sim d_{k(n)}$ with $d_{n}=n^{\log \beta/\log q}$, which is a regularly varying sequence of index $1/\alpha$. Further note that Theorem \ref{maxsdoa} does not exclude a max-stable limit with a constant function $\nu_{\beta,q}$. At least for the symmetric case $p=q=\frac12$ and $\beta=2^{-1/2}$ there is some (numerical) evidence in \cite{Sol2} of a max-stable limit. The following examples illustrate the occurrence of a constant function $\nu_{2^{-1/2},2^{-1}}$.
\end{remark}

\begin{example}\label{exsymm}
We consider the case $p=q=\frac12$ and $\beta=2^{-1/2}$ in \cite{Sol2} and for simplicity write $F=F_{2^{-1/2},2^{-1}}$. Note that $\frac1{\beta}=\sqrt{2}$ is a Garsia number. Assume that $\nu_{2^{-1/2},2^{-1}}\equiv a$ for some $a>0$, then $F(x)=a\,x^2$ for all $x\in(0,\sqrt{2}-1]$ by Lemma \ref{repstatdistr}. All we have to show is that the extension method leads to a continuous distribution function on $(0,1]$ that fulfills \eqref{dfstatsol} from which we receive $a$ by $F(1)=1$.\\
\underline{Step 1:} For $x\in(\sqrt{2}-1,\sqrt{2}(\sqrt{2}-1)]$, where $\sqrt{2}(\sqrt{2}-1)\approx 0.58<1$, we get by \eqref{dfstatsolalt}
\begin{align*}
F(x) & =2\,F(\tfrac{x}{\sqrt{2}})-F(x+1-\sqrt{2})\\
& =a\left(2\,(\tfrac{x}{\sqrt{2}})^2-(x+1-\sqrt{2})^2\right)\\
& =a\left(2(\sqrt{2}-1)x-(\sqrt{2}-1)^2\right).
\end{align*}
\underline{Step 2:} For $x\in(\sqrt{2}(\sqrt{2}-1),2(\sqrt{2}-1)]$, where $2(\sqrt{2}-1)\approx 0.828<1$, we have $x+1-\sqrt{2}\leq\sqrt{2}-1$ and hence we get by \eqref{dfstatsolalt}
\begin{align*}
F(x) & =2\,F(\tfrac{x}{\sqrt{2}})-F(x+1-\sqrt{2})\\
& =a\left(2\,(\sqrt{2}(\sqrt{2}-1)x-(\sqrt{2}-1)^2)-(x+1-\sqrt{2})^2\right)\\
& =a\left(2x-x^2-3(\sqrt{2}-1)^2\right).
\end{align*}
\underline{Step 3:} Since $\sqrt{2}^3(\sqrt{2}-1)\approx 1.171>1$ we get for $x\in(2(\sqrt{2}-1),1]$ using \eqref{dfstatsolalt} together with $x+1-\sqrt{2}\in(\sqrt{2}-1,\sqrt{2}(\sqrt{2}-1)]$ 
\begin{align*}
F(x) & =2\,F(\tfrac{x}{\sqrt{2}})-F(x+1-\sqrt{2})\\
& =a\left(2\,(\sqrt{2}x-\tfrac12 x^2-3(\sqrt{2}-1)^2)-2(\sqrt{2}-1)(x+1-\sqrt{2})+(\sqrt{2}-1)^2\right)\\
& =a\left(2x-x^2-3(\sqrt{2}-1)^2\right),
\end{align*}
which coincides with the representation in Step 2. It can easily be verified that $F$ is indeed a continuous distribution function on $(0,1]$ that fulfills \eqref{dfstatsol} by choosing $a^{-1}=1-3(\sqrt{2}-1)^2$. We recover the density in \cite{Sol2} from $f=F'$ given by
$$f(x)=\begin{cases}
\frac{2x}{1-3(\sqrt{2}-1)^2} & \text{ if } x\in(0,\sqrt{2}-1),\\
\frac{2(\sqrt{2}-1)}{1-3(\sqrt{2}-1)^2} & \text{ if } x\in(\sqrt{2}-1,\sqrt{2}(\sqrt{2}-1)),\\
\frac{2-2x}{1-3(\sqrt{2}-1)^2} & \text{ if } x\in(\sqrt{2}(\sqrt{2}-1),1).
\end{cases}$$
Note that this extends to a continuous density function that is symmetric with respect to the axis $x=\frac12$ due to $F(x)=1-F(1-x)$.
\end{example}

\begin{example}\label{exsymmho}
As a generalization of Example \ref{exsymm} we may also check if for $p=\frac12$ and $\beta=2^{-1/k}$ with $k\geq3$ (these correspond to Garsia numbers $\frac1{\beta}$) the extension method leads to a max-stable limit. In \cite{Win} it was shown that the stationary distribution function belongs to $C^k(0,1)$. For simplicity we consider $k=3$ and write $F=F_{2^{-1/3},2^{-1}}$. If $\nu_{2^{-1/3},2^{-1}}\equiv a$, then by Lemma \ref{repstatdistr} we have $F(x)=a\,x^3$ for all $x\in(0,2^{1/3}-1]$. As above, with some elementary calculations the extension method gives us:\\
\underline{Step 1:} For $x\in(2^{1/3}-1,2^{1/3}(2^{1/3}-1)]$ we have
$$F(x)=a\big(3(2^{1/3}-1)x^2-3(2^{1/3}-1)^2 x+(2^{1/3}-1)^3\big).$$
\underline{Step 2:} For $x\in(2^{1/3}(2^{1/3}-1),2^{2/3}(2^{1/3}-1)]$ we have
$$F(x)=a\big((\tfrac{6}{2^{2/3}}+3)(2^{1/3}-1)x^2-x^3-(\tfrac{6}{2^{1/3}}+3)(2^{1/3}-1)^2 x+3(2^{1/3}-1)^3\big).$$
\underline{Step 3:} For $x\in(2^{2/3}(2^{1/3}-1),2(2^{1/3}-1)]$ we have
$$F(x)=a\big((\tfrac{6}{2^{1/3}}+3\cdot 2^{1/3}+3)(2^{1/3}-1)x^2-2\,x^3-(6\cdot 2^{1/3}+3\cdot 2^{2/3}+3)(2^{1/3}-1)^2 x+7(2^{1/3}-1)^3\big).$$
Note that $2(2^{1/3}-1)>\frac12$ and we may calculate $a$ by $F(\frac12)=\frac12$ and use the reflection method $F(x)=1-F(1-x)$ to extend $F$ to a continuous distribution function on $(0,1)$. However, it remains unclear if the assumption of a constant function $\nu_{2^{-1/k},2^{-1}}$ indeed leads to a distribution function $F_{2^{-1/k},2^{-1}}$ that fulfills \eqref{dfstatsol}. Figure \ref{fig:garsia} shows numerical approximations of the functions $F_{\beta,p}$ and $\nu_{\beta,p}$ for $\beta \in \{2^{-1/3}, 2^{-1/4}, 2^{-1/5}\}$ and $p=\frac12$; in these cases the function $\nu_{\beta,p}$ seems to be constant.
\end{example}

\begin{figure}
\includegraphics{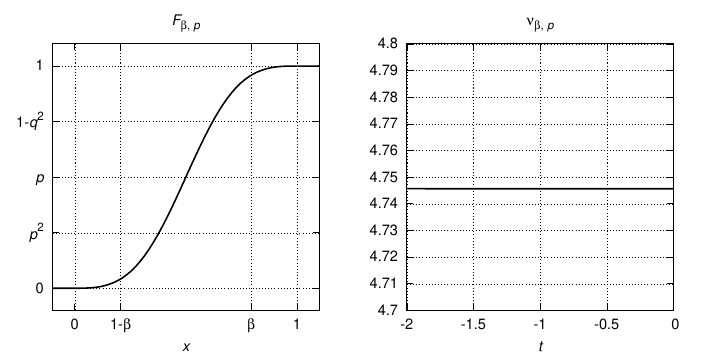}
\includegraphics{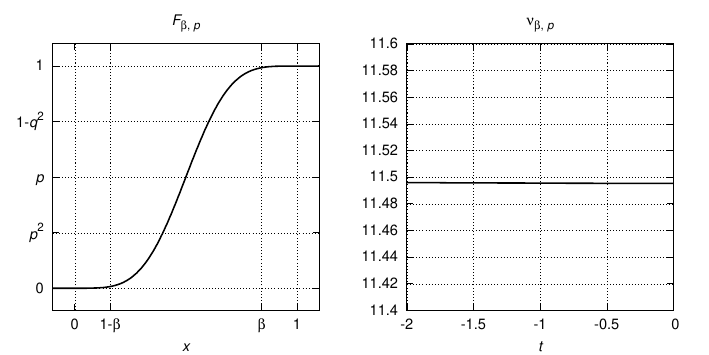}
\includegraphics{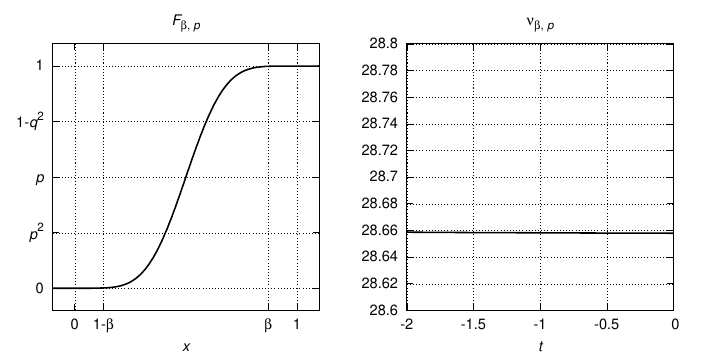}
\caption{\small As Figure~\ref{fig:goldenratio}, but for the parameters $\beta \in \{2^{-1/3}, 2^{-1/4}, 2^{-1/5}\}$ (top to bottom) and $p = \frac12$ in which case the function $\nu_{\beta,p}$ seems to be constant.}
\label{fig:garsia}
\end{figure}

\section{The merge theorem}

In Section 4 we will prove that for $\beta\in(\frac12,1)$ the max-semistable limit theorem for i.i.d.\ random variables $(Y_{k})_{k\in\nat_{0}}$ in Theorem \ref{maxsdoa} has an analogue for the $\AR$-sequence $(X_{k})_{k\in\nat_{0}}$, similarly to what is shown in \cite{Sterk} for the case $\beta\in(0,\frac12)$. In fact we will prove a stronger merge theorem for all $\beta\in(0,1)$, which we first present in the i.i.d.\ situation in this section. Let
\begin{equation}
H_{\beta,q}(x)=\begin{cases}\exp\left(-(-x)^{\log q/\log\beta}\nu_{\beta,q}^{\to}(\log(-x))\right) & \text{ for }x<0,\\
1 & \text{ for }x\geq0, \end{cases}
\end{equation}
be the (continuous) limiting max-semistable distribution function in Theorem \ref{maxsdoa} for $\beta\in(\frac12,1)$, respectively from Theorem 3.1 in \cite{Sterk} for $\beta\in(0,\frac12]$ with $\nu_{\beta,q}^{\to}=\nu_{\beta,q}$ (an extension is not needed here). Recall the normalizing sequences $a_{n}=\beta^{-n}$, $b_{n}\equiv1$ and the subsequence $k(n)=\lfloor q^{-n}\rfloor$ from Theorem \ref{maxsdoa}. For $n\geq k(1)$ write $n=k(m_{n})r_{n}$ with $m_{n}\in\nat$ and $r_{n}\geq1$, uniquely given by $k(m_{n})\leq n<k(m_{n}+1)$. Then the sequence $(r_{n})_{n>k(1)}$ is relatively compact in $[1,c]=[1,q^{-1}]$. Then with the modified normalizing sequence $(a_{n}^{\circ}=a_{m_{n}})_{n\in\nat}$ we get for the maxima $M_{n}^{\circ}=\max\{Y_{0},\ldots,Y_{n-1}\}$ of an i.i.d.\ sequence $(Y_{k})_{k\in\nat_{0}}$ with common distribution function $F_{\beta,p}$ by Theorem 4 in \cite{Megyesi}:

\begin{theorem}[\bf Merge theorem]\label{merge}
With the above notations we get for all $\beta\in(0,1)$ as $n\to\infty$
\begin{equation}\label{mergecomp}
\sup_{x\in\rr}\left|\Pm\left(a_{n}^{\circ}(M_{n}^{\circ}-1)\leq x\right)-H^{r_{n}}_{\beta,q}(x)\right|\to0.
\end{equation}
\end{theorem}

Compared to \cite{Megyesi}, the proof of Theorem \ref{merge} simplifies in our situation due to the fact that the limit distribution function $H_{\beta,q}$ is continuous. We give this simplified proof because its steps will be useful to transfer the result to the $\AR$ case in Section 4.

\begin{proof}[Proof of Theorem \ref{merge}]
It suffices to prove
\begin{equation}\label{mergesuff}
\sup_{x\leq0}\left|\Pm\left(a_{n}^{\circ}(M_{n}^{\circ}-1)\leq x\right)-H^{r_{n}}_{\beta,q}(x)\right|\to0.
\end{equation}
\underline{Step 1:} Consider a subsequence $n'=k(m_{n'})r_{n'}$ such that $r_{n'}\to r$ for some $r\in[1,c]$. Then for any sequence $(x_{n'})\subseteq(\-\infty,0]$ with $x_{n'}\to x\leq0$ we get as in the proof of Theorem \ref{maxsdoa}
\begin{align*}
\Pm\left(a_{n'}^{\circ}(M_{n'}^{\circ}-1)\leq x_{n'}\right) & =\Pm\left(a_{m_{n'}}(M_{k(m_{n'})r_{n'}}^{\circ}-1)\leq x_{n'}\right) =\Pm(Y_{0}\leq a_{m_{n'}}^{-1}x_{n'}+1)^{k(m_{n'})r_{n'}}\\
& =F_{\beta,p}(a_{m_{n'}}^{-1}x_{n'}+1)^{k(m_{n'})r_{n'}}=\left(1-F_{\beta,q}(-a_{m_{n'}}^{-1}x_{n'})\right)^{k(m_{n'})r_{n'}}\\
& =\left(1-(-a_{m_{n'}}^{-1}x_{n'})^{\log q/\log\beta}\nu_{\beta,q}(\log(-a_{m_{n'}}^{-1}x_{n'}))\right)^{k(m_{n'})r_{n'}}\\
& =\exp\left(k(m_{n'})\log\left(1-(-a_{m_{n'}}^{-1}x_{n'})^{\log q/\log\beta}\nu_{\beta,q}(\log(-a_{m_{n'}}^{-1}x_{n'}))\right)\right)^{r_{n'}}\\
& \sim\exp\left(-k(m_{n'})(-a_{m_{n'}}^{-1}x_{n'})^{\log q/\log\beta}\nu_{\beta,q}^{\to}(\log(-x_{n'}))\right)^{r_{n'}}\\
& =\exp\left(-\lfloor q^{-m_{n'}}\rfloor q^{m_{n'}}(-x_{n'})^{\log q/\log\beta}\nu_{\beta,q}^{\to}(\log(-x_{n'}))\right)^{r_{n'}}\\
& \to\exp\left(-(-x)^{\log q/\log\beta}\nu_{\beta,q}^{\to}(\log(-x))\right)^{r}=H_{\beta,q}^{r}(x),
\end{align*}
since $\nu_{\beta,q}^{\to}$ is a $\log(1/\beta)$-periodic and continuous function.\\[1ex]

\underline{Step 2:} Let $F_{n'}(x)=\Pm\left(a_{n'}^{\circ}(M_{n'}^{\circ}-1)\leq x\right)$. Since $H_{\beta,q}$ is continuous, it follows from step 1 that $F_{n'}(x)\to H_{\beta,q}^{r}(x)$ uniformly on compact subsets of $x\leq0$, whenever $(n')$ is a subsequence such that $n'=k(m_{n'})r_{n'}$ with $r_{n'}\to r\in[1,c]$. For $\eps>0$ choose $M>0$ such that for all $x<-M$ we have $F_{n'}(x)<\eps/2$ (for all $n'$ by Prohorov's theorem) and $H_{\beta,q}^{r}(x)<\eps/2$. Then we get
$$\limsup_{n'\to\infty}\sup_{x\leq0}\left|F_{n'}(x)-H_{\beta,q}^{r}(x)\right|\leq\eps+\limsup_{n'\to\infty}\sup_{-M\leq x\leq0}\left|F_{n'}(x)-H_{\beta,q}^{r}(x)\right|=\eps,$$
which shows that
\begin{equation}\label{mergesuff1}
\sup_{x\leq0}\left|F_{n'}(x)-H^{r}_{\beta,q}(x)\right|\to0
\end{equation}
as $n'=k(m_{n'})r_{n'}\to\infty$, whenever $r_{n'}\to r\in[1,c]$.\\[1ex]

\underline{Step 3:} With the same arguments as in step 2, we get
\begin{equation*}
\sup_{x\leq0}\left|H^{r_{n'}}_{\beta,q}(x)-H^{r}_{\beta,q}(x)\right|\to0,
\end{equation*}
whenever $r_{n'}\to r\in[1,c]$. Combining this with \eqref{mergesuff1} we get
\begin{equation}\label{mergesuff2}
\sup_{x\leq0}\left|F_{n'}(x)-H^{r_{n'}}_{\beta,q}(x)\right|\to0
\end{equation}
as $n'=k(m_{n'})r_{n'}\to\infty$, whenever $r_{n'}\to r\in[1,c]$. Since every subsequence of the natural numbers $n=k(m_{n})r_{n}$ contains a further subsequence $n'$ such that $r_{n'}\to r\in[1,c]$ and hence \eqref{mergesuff2} holds, we conclude that \eqref{mergesuff} is valid.
\end{proof}

By Corollary 1 in \cite{Megyesi} it is also possible to modify the normalizing sequence $(a_{n}^{\circ})_{n\in\nat}$ in such a way that the corresponding limiting distribution function as $r_{n'}\to1$ and $r_{n'}\to c$ are the same. This is a consequence of the fact that $H_{\beta,q}$ and $H_{\beta,q}^{c}$ are of the same type due to the periodicity of $\nu_{\beta,q}^{\to}$. In our situation we take the modified sequence $a_{n}^{\ast}=r_{n}^{\log\beta/\log q}a_{n}^{\circ}$, then as $r_{n'}\to r\in[1,c]$ we get as in step 1 of the proof of Theorem \ref{merge} for all $x\leq 0$
\begin{align*}
\Pm\left(a_{n'}^{\ast}(M_{n'}^{\circ}-1)\leq x\right) & =\Pm\left(a_{n'}^{\circ}(M_{n'}^{\circ}-1)\leq r_{n}^{-\log\beta/\log q}x\right)\\
& \to H_{\beta,q}^{r}(r^{-\log\beta/\log q}x)\\
& =\exp\left(-r(-r^{-\log\beta/\log q}x)^{\log q/\log\beta}\nu_{\beta,q}^{\to}(\log(-r^{-\log\beta/\log q}x))\right)\\
& =\exp\left(-(-x)^{\log q/\log\beta}\nu_{\beta,q}^{\to}\left(-\tfrac{\log\beta}{\log q}\,\log r+\log(-x)\right)\right)=: G_{r}(x).
\end{align*}
In case $r=c=q^{-1}$ we get that $G_{c}(x)=G_{1}(x)$ for all $x\leq0$ by the $\log(1/\beta)$-periodicity of $\nu_{\beta,q}^{\to}$.

\section{Extremes of the $\AR$-process}

We follow the outline given in \cite{Sterk} based on techniques in \cite{Cher} and \cite{GMS} to extend the results to a merge theorem for the full parameter range $\beta\in(0,1)$. We write again $n=k(m_{n})r_{n}$ with $k(m_{n})\leq n<k(m_{n}+1)$ and $a_{n}^{\circ}=\beta^{-m_{n}}$. Consider the extremes of the $\AR$ sequence
$$M_{n}=\max\{X_{0},\ldots,X_{n-1}\},$$
where $X_{0}$ has the stationary distribution function $F_{\beta,p}$. For fixed $x<0$ we define
$$u_{n}(x)=1+\beta^{n}x$$
and
\begin{align*}
j_{n}(\beta) & =\begin{cases}\max\{i\in\ganz:u_{n-i}(x)>0\} & \text{ if } \beta\in(0,\frac12],\\
\max\{i\in\ganz:u_{n-i}(x)>2-\frac1{\beta}\} & \text{ if }\beta\in(\frac12,1),\end{cases}\\
& =\begin{cases}\max\{i\in\ganz:i<n+\frac{\log(-x)}{\log\beta}\} & \text{ if } \beta\in(0,\frac12],\\
\max\{i\in\ganz:i<n+\frac{\log(-x)-\log(\frac1{\beta}-1)}{\log\beta}\} & \text{ if } \beta\in(\frac12,1).\end{cases}
\end{align*}
Note that $(u_{n}(x))_{n\in\nat}$ is strictly increasing and $\frac{j_{n}(\beta)}{n}\to1$. We only consider sufficiently large values of $n\in\nat$ such that for a sequence $(x_{n})_{n\in\nat}\subseteq(-\infty,0)$ with $x_{n}\to x<0$ we have $u_{n}(x_{n})\in(0,1]$ and $j_{n}(\beta)\in\nat$, where $(j_{n}(\beta)-n)_{n\in\nat}$ is bounded. For all $i\in\{0,\ldots,j_{n}(\beta)\}$ we get using \eqref{dfreverse}
\begin{align*}
F_{\beta,p}(u_{n-i}(x)) & =F_{\beta,p}(1+\beta^{n-i}x)=1-F_{\beta,q}(-\beta^{n-i}x)\\
& =1-(-\beta^{n-i}x)^{\log q/\log\beta}\nu_{\beta,q}(\log(-\beta^{n-i}x)).
\end{align*}
By definition of $j_{n}(\beta)$ we have
$$-\beta^{n-i}x=1-u_{n-i}(x)<\begin{cases}1 & \text{ if } \beta\in(0,\frac12],\\
\frac1{\beta}-1 & \text{ if } \beta\in(\frac12,1),\end{cases}$$
and can apply Proposition 2.1 in \cite{Sterk}, respectively Lemma \ref{repstatdistr}. By the $\log(1/\beta)$-periodicity of $\nu_{\beta,q}^{\to}$ we get for all $i\in\{0,\ldots,j_{n}(\beta)\}$
\begin{equation}\label{distrfirst}\begin{split}
F_{\beta,p}(u_{n-i}(x)) & =1-q^{n-i}(-x)^{\log q/\log\beta}\nu_{\beta,q}^{\to}(\log(-x))\\
& =:1-q^{n-i}\psi_{\beta,q}(x).
\end{split}\end{equation}
We further define for $s\in\nat$ and $i\in\{0,\ldots,j_{n}(\beta)\}$
$$P_{s,i}:=\Pm(M_{s}\leq u_{n}(x), X_{s-1}\leq u_{n-i}(x))$$
with the special cases
\begin{align}
P_{1,i} & =\Pm(X_{0}\leq u_{n}(x), X_{0}\leq u_{n-i}(x))=\Pm(X_{0}\leq u_{n-i}(x)) \label{Pspec1},\\
P_{s,0} & =\Pm(M_{s}\leq u_{n}(x), X_{s-1}\leq u_{n}(x))=\Pm(M_{s}\leq u_{n}(x)) \label{Pspec2}.
\end{align}
The following Lemma is proven in Lemma 3.1 of \cite{Sterk}:

\begin{lemma}\label{Psequations}
{\rm (a)} For all $s\geq2$ and $i\in\{0,\ldots,j_{n}(\beta)-1\}$ we have the recursion formula
$$P_{s,i}=p\,P_{s-1,0}+q\, P_{s-1,i+1}.$$
{\rm (b)} For all $s\in\{2,\ldots,j_{n}(\beta)+1\}$ we have
$$P_{s,0}=\sum_{i=0}^{s-2}p\,q^{i}P_{1,i}+q^{s-1}P_{1,s-1}.$$
\end{lemma}

Part (a) of the following Lemma is proven in Lemma 3.2 of \cite{Sterk}.

\begin{lemma}\label{maxconvmain}
{\rm (a)} For all $s\in\{2,\ldots,j_{n}(\beta)+1\}$ we have
$$\Pm(M_{s}\leq u_{n}(x))=1-\big(p(s-1)+1\big)q^{n}\psi_{\beta,q}(x).$$
{\rm (b)} For every subsequence $n'=k(m_{n'})r_{n'}$ such that $r_{n'}\to r\in[1,c]$ we have as $n'\to\infty$
$$\Pm\big(M_{j_{m_{n'}}(\beta)+1}\leq u_{m_{n'}}(x_{n'})\big)^{\lfloor\frac{n'}{j_{m_{n'}}(\beta)}\rfloor}\to\exp(-pr\,\psi_{\beta,q}(x))$$
for any sequence $(x_{n'})\subseteq(-\infty,0)$ with $x_{n'}\to x<0$.
\end{lemma}

\begin{proof}
We only have to prove part (b). For $s=j_{m_{n'}}(\beta)+1$ we get from part (a)
\begin{align*}
\Pm\big(M_{j_{m_{n'}}(\beta)+1}\leq u_{m_{n'}}(x_{n'})\big)^{\lfloor\frac{n'}{j_{m_{n'}}(\beta)}\rfloor} & =\big(1-(p\,j_{m_{n'}}(\beta)+1)q^{m_{n'}}\psi_{\beta,q}(x_{n'})\big)^{\lfloor\frac{n'}{j_{m_{n'}}(\beta)}\rfloor}\\
& =\left(1-\frac{(p\,j_{m_{n'}}(\beta)+1)q^{m_{n'}}\lfloor\frac{n'}{j_{m_{n'}}(\beta)}\rfloor\psi_{\beta,q}(x_{n'})}{\lfloor\frac{n'}{j_{m_{n'}}(\beta)}\rfloor}\right)^{\lfloor\frac{n'}{j_{m_{n'}}(\beta)}\rfloor}\\
&\to\exp(-pr\,\psi_{\beta,q}(x)),
\end{align*}
since $\psi_{\beta,q}$ is continuous and
$$(p\,j_{m_{n'}}(\beta)+1)q^{m_{n'}}\lfloor\tfrac{n'}{j_{m_{n'}}(\beta)}\rfloor=\frac{p\,j_{m_{n'}}(\beta)+1}{j_{m_{n'}}(\beta)}q^{m_{n'}}\lfloor\tfrac{k(m_{n'})r_{n'}}{j_{m_{n'}}(\beta)}\rfloor\frac{j_{m_{n'}}(\beta)}{k(m_{n'})}\lfloor q^{-m_{n'}}\rfloor\to pr,$$
where $j_{n}(\beta)\to\infty$ with $\frac{k(n)}{j_{n}(\beta)}=\frac{\lfloor q^{-n}\rfloor}{j_{n}(\beta)}\to\infty$.
\end{proof}

Note that the limit in Lemma \ref{maxconvmain}(b) is
\begin{equation}\label{limitH}
\exp(-pr\,\psi_{\beta,q}(x))=\exp\big((-x)^{\log q/\log\beta}\nu_{\beta,q}^{\to}(\log(-x))\big)^{pr}=H_{\beta,q}^{pr}(x).
\end{equation}

\begin{theorem}\label{mergeAR1sub}
Let $n'=k(m_{n'})r_{n'}$ be a subsequence such that $r_{n'}\to r\in[1,c]$. Then for  any sequence $(x_{n'})\subseteq(-\infty,0]$ with $x_{n'}\to x<0$ we have
$$\Pm(a_{n'}^{\circ}(M_{n'}-1)\leq x_{n'})\to H_{\beta,q}^{pr}(x).$$
\end{theorem}

As in the proof of the merge theorem, we can conclude the following result directly from Theorem \ref{mergeAR1sub} by carrying out Step 2 and Step 3 in the proof of Theorem \ref{merge}.

\begin{corollary}[\bf Merge theorem for the $\AR$]\label{mergeAR1}
As $n\to\infty$ we have
\begin{equation}\label{mergeARcomp}
\sup_{x\in\rr}\big|\Pm(a_{n}^{\circ}(M_{n}-1)\leq x)-H_{\beta,q}^{pr_{n}}(x)\big|\to0.
\end{equation}
In particular we get for the subsequence $(k(n))_{n\in\nat}$ as $n\to\infty$
$$\sup_{x\in\rr}\big|\Pm(a_{n}(M_{k(n)}-1)\leq x)-H_{\beta,q}^{p}(x)\big|\to0.$$
\end{corollary}

\begin{remark}
Comparing \eqref{mergeARcomp} with the i.i.d.\ result \eqref{mergecomp}, according to page 67 in \cite{LLR}, the $\AR$-process has extremal index $p$. This extremal index can be interpreted in several ways, in particular as
\begin{itemize}
\item the loss of i.i.d.\ degrees of freedom,
\item the reciprocal of the expected cluster size of the $\AR$-process,
\item or recurrences near unstable fixed points of corresponding dynamical systems;
\end{itemize}
see \cite{MFS} for details.
\end{remark}

\begin{proof}[Proof of Theorem \ref{mergeAR1sub}]
Since we have
\begin{align*}
\Pm(a_{n'}^{\circ}(M_{n'}-1)\leq x_{n'}) & 
=\Pm(M_{n'}\leq 1+\beta^{m_{n'}}x_{n'})=\Pm(M_{n'}\leq u_{m_{n'}}(x_{n'})),
\end{align*}
in view of Lemma \ref{maxconvmain}(b) and \eqref{limitH} it remains to show
\begin{equation}\label{convleft}
\left|\Pm(M_{n'}\leq u_{m_{n'}}(x_{n'}))-\Pm\big(M_{j_{m_{n'}}(\beta)+1}\leq u_{m_{n'}}(x_{n'})\big)^{\lfloor\frac{n'}{j_{m_{n'}}(\beta)}\rfloor}\right|\to0.
\end{equation}
To prove \eqref{convleft} we write $s_{n'}:=\lfloor\frac{n'}{j_{m_{n'}}(\beta)}\rfloor$ and divide the integers $\{0,\ldots,s_{n'}j_{m_{n'}}(\beta)\}$ into intervals
\begin{align*}
I_{i} & =\{(i-1)j_{m_{n'}}(\beta),\ldots,i\,j_{m_{n'}}(\beta)-\ell_{n'}-1\}\\
I_{i}^{\ast} & =\{i\,j_{m_{n'}}(\beta)-\ell_{n'},\ldots,i\,j_{m_{n'}}(\beta)-1\}
\end{align*}
for $i\in\{1,\ldots,s_{n'}\}$, where the intervals $I_{i}$ all have length $j_{m_{n'}}(\beta)-\ell_{n'}$, separated by gaps $I_{i}^{\ast}$ of length $\ell_{n'}$. We will choose the sequence $(\ell_{n'})$ such that $\ell_{n'}\to\infty$ and $\ell_{n'}/j_{m_{n'}}(\beta)\to0$. Then we have
\begin{equation}\label{ellasymp}
j_{m_{n'}}(\beta)q^{m_{n'}}\to0\quad\text{ and }\quad s_{n'}\ell_{n'}q^{m_{n'}}\sim n'q^{m_{n'}}\,\frac{\ell_{n'}}{j_{m_{n'}}(\beta)}\sim r_{n'}\,\frac{\ell_{n'}}{j_{m_{n'}}(\beta)}\to0.
\end{equation}
For any set $I$ of nonnegative integers we write $M(I)=\max\{X_{i}: i\in I\}$, then we have
$$\left|\Pm(M_{n'}\leq u_{m_{n'}}(x_{n'}))-\Pm\big(M_{j_{m_{n'}}(\beta)+1}\leq u_{m_{n'}}(x_{n'})\big)^{s_{n'}}\right|\leq T_{1}+T_{2}+T_{3}+T_{4},$$
where
\begin{align*}
T_{1} & =\left|\Pm(M_{n'}\leq u_{m_{n'}}(x_{n'}))-\Pm(M_{j_{m_{n'}}(\beta)s_{n'}}\leq u_{m_{n'}}(x_{n'}))\right|\\
T_{2} & =\left|\Pm(M_{j_{m_{n'}}(\beta)s_{n'}}\leq u_{m_{n'}}(x_{n'}))-\Pm\left(\bigcap_{i=1}^{s_{n'}}\{M(I_{i})\leq u_{m_{n'}}(x_{n'})\}\right)\right|\\
T_{3} & =\left|\Pm\left(\bigcap_{i=1}^{s_{n'}}\{M(I_{i})\leq u_{m_{n'}}(x_{n'})\}\right)-\prod_{i=1}^{s_{n'}}\Pm\big(M(I_{i})\leq u_{m_{n'}}(x_{n'})\big)\right|\\
T_{4} & =\left|\prod_{i=1}^{s_{n'}}\Pm\big(M(I_{i})\leq u_{m_{n'}}(x_{n'})\big)-\Pm\big(M_{j_{m_{n'}}(\beta)+1}\leq u_{m_{n'}}(x_{n'})\big)^{s_{n'}}\right|
\end{align*}
and \eqref{convleft} is proven if we show $T_{1},\ldots,T_{4}\to0$.\\[1ex]
\underline{Proof of $T_{1}\to0$:} We have $n'-j_{m_{n'}}(\beta)\leq j_{m_{n'}}(\beta)s_{n'}\leq n'$ and thus we get using \eqref{distrfirst}
\begin{align*}
0 & \leq \Pm(M_{j_{m_{n'}}(\beta)s_{n'}}\leq u_{m_{n'}}(x_{n'}))-\Pm(M_{n'}\leq u_{m_{n'}}(x_{n'}))\\
& =\Pm(X_{i}>u_{m_{n'}}(x_{n'})\text{ for some }i\in\{j_{m_{n'}}(\beta)s_{n'}+1,\ldots,n'\})\\
& \leq(n'-j_{m_{n'}}(\beta)s_{n'})(1-F_{\beta,p}(u_{m_{n'}}(x_{n'})))\\
&\leq j_{m_{n'}}(\beta)q^{m_{n'}}\psi_{\beta,q}(x_{n'})\to0,
\end{align*}
since $\psi_{\beta,q}(x_{n'})\to\psi_{\beta,q}(x)\geq0$ and \eqref{ellasymp} holds. This shows that $T_{1}\to0$.\\[1ex]
\underline{Proof of $T_{2}\to0$:} We have using \eqref{distrfirst}
\begin{align*}
0 & \leq\Pm\left(\bigcap_{i=1}^{s_{n'}}\{M(I_{i})\leq u_{m_{n'}}(x_{n'})\}\right)-\Pm(M_{j_{m_{n'}}(\beta)s_{n'}}\leq u_{m_{n'}}(x_{n'}))\\
& =\Pm\left(X_{i}>u_{m_{n'}}(x_{n'})\text{ for some }i\in I_{1}^{\ast}\cup\cdots\cup I_{s_{n'}^{\ast}}\right)\\
& \leq s_{n'}\ell_{n'}(1-F_{\beta,p}(u_{m_{n'}}(x_{n'})))\\
& = s_{n'}\ell_{n'}q^{m_{n'}}\psi_{\beta,q}(x_{n'})\to0,
\end{align*}
since $\psi_{\beta,q}(x_{n'})\to\psi_{\beta,q}(x)\geq0$ and \eqref{ellasymp} holds. This shows that $T_{2}\to0$.\\[1ex]
\underline{Proof of $T_{4}\to0$:} For $0\leq x\leq y\leq1$ and $s>1$ we have $0\leq y^{s}-x^{s}\leq s(y-x)$ by the mean value theorem. Hence we get by stationarity of the $\AR$ sequence
\begin{align*}
0 & \leq T_{4}=\Pm\big(M(I_{1})\leq u_{m_{n'}}(x_{n'})\big)^{s_{n'}}-\Pm\big(M_{j_{m_{n'}}(\beta)+1}\leq u_{m_{n'}}(x_{n'})\big)^{s_{n'}}\\
& \leq s_{n'}\left(\Pm\big(M(I_{1})\leq u_{m_{n'}}(x_{n'})\big)-\Pm\big(M_{j_{m_{n'}}(\beta)+1}\leq u_{m_{n'}}(x_{n'})\big)\right)\\
& \leq s_{n'}\Pm(X_{i}>u_{m_{n'}}(x_{n'})\text{ for some }i\in\{j_{m_{n'}}(\beta)-\ell_{n'},\ldots,j_{m_{n'}}(\beta)\})\\
& \leq s_{n'}(\ell_{n'}+1)q^{m_{n'}}\psi_{\beta,q}(x_{n'})\to0,
\end{align*}
as in the proof of $T_{2}\to0$.\\[1ex]
\underline{Proof of $T_{3}\to0$:} In Appendix A of \cite{Sterk} it is shown that the events $A_{s}=\{M(I_{s})\leq u_{m_{n'}}(x_{n'})\}$ fulfill
$$0\leq\Pm\Big(\bigcap_{s=1}^{s_{n'}}A_{s}\Big)-\prod_{s=1}^{s_{n'}}\Pm(A_{s})\leq\sum_{s=2}^{s_{n'}}\left(\Pm\Big(A_{s}\Big|\bigcap_{i=1}^{s-1}A_{i}\Big)-\Pm(A_{s})\right).$$
Since the intervals $I_s$ are disjoint and all have the same length, by the Markov property and the stationarity of the $\AR$-process we have
$$\Pm\Big(A_{s}\Big|\bigcap_{i=1}^{s-1}A_{i}\Big)-\Pm(A_{s})=\Pm(A_{s}\mid A_{s-1})-\Pm(A_{s})=\Pm(A_{2}\mid A_{1})-\Pm(A_{2}).$$
Below we will show that 
\begin{equation}\label{remain}
0\leq\Pm(A_{2}\mid A_{1})-\Pm(A_{2})\leq C(j_{m_n'}(\beta)-\ell_{n'})q^{2m_{n'}-j_{m_{n'}}(\beta)+\ell_{n'}}
\end{equation}
for some constant $C>0$. Since $s_{n'}\sim\frac{n'}{j_{m_{n'}}(\beta)}\sim\frac{q^{-m_{n'}}r_{n'}}{j_{m_{n'}}(\beta)}$ this results in
$$0\leq T_{3}\leq C\,s_{n'}(j_{m_n'}(\beta)-\ell_{n'})q^{2m_{n'}-j_{m_{n'}}(\beta)+\ell_{n'}}\sim C\,r_{n'}q^{m_{n'}-j_{m_{n'}}(\beta)+\ell_{n'}}\sim C'\,q^{\ell_{n'}}\to0$$
for some constant $C'>0$ due to $\ell_{n'}\to\infty$, $\frac{\ell_{n'}}{j_{m_{n'}}(\beta)}\to0$ and the boundedness of the sequences $(r_{n'})_{n'\in\nat}$, $(m_{n'}-j_{m_{n'}}(\beta))_{n'\in\nat}$. It remains to show \eqref{remain}.\\[1ex]
Note that 
\begin{align*}
A_1 & =\{M(I_1)\leq u_{m_{n'}}(x_{n'})\}=\{X_0\leq u_{m_{n'}}(x_{n'}),\ldots,X_{j_{m_{n'}}(\beta)-\ell_{n'}-1}\leq u_{m_{n'}}(x_{n'})\}
\intertext{and}
A_2 & =\{M(I_2)\leq u_{m_{n'}}(x_{n'})\}=\{X_{j_{m_{n'}}(\beta)}\leq u_{m_{n'}}(x_{n'}),\ldots,X_{2j_{m_{n'}}(\beta)-\ell_{n'}-1}\leq u_{m_{n'}}(x_{n'})\}.
\end{align*}
Using regular conditional probability for the measure $\Qm:=\Pm(\cdot\mid A_1)$ we get
\begin{equation}\label{regcondprob}\begin{split}
\Pm(A_2\mid A_1) & =\Qm(A_2)=\int_0^{u_{m_{n'}}(x_{n'})}\Qm(A_2\mid X_{j_{m_{n'}}(\beta)-\ell_{n'}-1}=v)\,d \Qm_{X_{j_{m_{n'}}(\beta)-\ell_{n'}-1}}(v)\\
& =\int_0^{u_{m_{n'}}(x_{n'})}\Qm(A_2\mid X_{j_{m_{n'}}(\beta)-\ell_{n'}-1}=v)\,dG(v)\\
& =\int_0^{u_{m_{n'}}(x_{n'})}\Pm(A_2\mid X_{j_{m_{n'}}(\beta)-\ell_{n'}-1}=v)\,dG(v),
\end{split}\end{equation}
where
$$G(v)=\Pm\big(X_{j_{m_{n'}}(\beta)-\ell_{n'}-1}\leq v\,\big|\,X_0\leq u_{m_{n'}}(x_{n'}),\ldots,X_{j_{m_{n'}}(\beta)-\ell_{n'}-1}\leq u_{m_{n'}}(x_{n'})\big)$$
and the last equality in \eqref{regcondprob} follows by the Markov property of the $\AR$. For all integers $k\geq0$ and $i\geq1$ we have by \eqref{statsol}
\begin{equation}\label{statsolrep}\begin{split}
X_{k+i} & =\sum_{j=0}^\infty \beta^j \eps_{k+i-j}=\beta^i\sum_{j=0}^\infty \beta^{j-i} \eps_{k+i-j}\\
& =\beta^i\sum_{j=i}^\infty \beta^{j-i} \eps_{k+i-j}+\beta^i\sum_{j=0}^{i-1} \beta^{j-i} \eps_{k+i-j}\\
& =\beta^i\sum_{j=0}^\infty \beta^{j} \eps_{k-j}+\sum_{j=0}^{i-1} \beta^{j} \eps_{k+i-j}=\beta^i X_k+Y_{i,k},
\end{split}\end{equation}
where $Y_{i,k}=\sum_{j=0}^{i-1} \beta^{j} \eps_{k+i-j}$ is independent of $X_0,\ldots,X_k$.
Further we have $Y_{i,k}\leq\sum_{j=0}^{i-1} \beta^{j}(1-\beta)=1-\beta^i$ and hence
\begin{equation}\label{Yprop}
\Pm(Y_{i,k}\leq1-\beta^i)=1\quad,\quad\Pm(Y_{i,k}=1-\beta^i)=q^i\quad,\quad\Pm(Y_{i,k}<1-\beta^i)=1-q^i.
\end{equation}
Since we will only need the independence of $X_0$ and \eqref{Yprop}, we suppress the dependence on $k$ and simply write $Y_i=Y_{i,k}$. With the decomposition \eqref{statsolrep} the integrand in \eqref{regcondprob} can be rewritten as
\begin{equation}\label{regcond2}\begin{split}
& \Pm(A_2\mid X_{j_{m_{n'}}(\beta)-\ell_{n'}-1}=v)\\
& \quad=\Pm(X_{j_{m_{n'}}(\beta)}\leq u_{m_{n'}}(x_{n'}),\ldots,X_{2j_{m_{n'}}(\beta)-\ell_{n'}-1}\leq u_{m_{n'}}(x_{n'})\mid X_{j_{m_{n'}}(\beta)-\ell_{n'}-1}=v)\\
& \quad=\Pm(Y_{\ell_{n'}+1}\leq u_{m_{n'}}(x_{n'})-\beta^{\ell_{n'}+1}v,\ldots,Y_{j_{m_{n'}}(\beta)}\leq u_{m_{n'}}(x_{n'})-\beta^{j_{m_{n'}}(\beta)}v)
\end{split}\end{equation}
We partition the integration interval in \eqref{regcondprob} into $j_{m_{n'}}(\beta)-\ell_{n'}+1$ parts as
\begin{align*}
[0,u_{m_{n'}}(x_{n'})] = & [0,u_{m_{n'}-j_{m_{n'}}(\beta)}(x_{n'})] \cup (u_{m_{n'}-j_{m_{n'}}(\beta)}(x_{n'}),u_{m_{n'}-(j_{m_{n'}}(\beta)-1)}(x_{n'})] \cup\ldots\\
& \ldots\cup  (u_{m_{n'}-(\ell_{n'}+2)}(x_{n'}),u_{m_{n'}-(\ell_{n'}+1)}(x_{n'})] \cup (u_{m_{n'}-(\ell_{n'}+1)}(x_{n'}),u_{m_{n'}}(x_{n'})].
\end{align*}
Note that $u_{m_{n'}}(x_{n'})-\beta^iv<1-\beta^i$ if and only if $v>1+\beta^{m_{n'}-i}x_{n'}=u_{m_{n'}-i}(x_{n'})$. Hence, if $0\leq v\leq u_{m_{n'}-j_{m_{n'}}(\beta)}(x_{n'})$,the right-hand side of \eqref{regcond2} can be trivially bounded from above by 1. On the other hand, if $v>u_{m_{n'}-i}(x_{n'})$ for some $i\in\{\ell_{n'}+1,\ldots,j_{m_{n'}}(\beta)\}$, then the right-hand side of \eqref{regcond2} can be bounded from above by
\begin{align*}
& \Pm(Y_{i}\leq u_{m_{n'}}(x_{n'})-\beta^iv,\ldots,Y_{j_{m_{n'}}(\beta)}\leq u_{m_{n'}}(x_{n'})-\beta^{j_{m_{n'}}(\beta)}v)\\
& \quad\leq\Pm(Y_{i}\leq u_{m_{n'}}(x_{n'})-\beta^iv)\leq\Pm(Y_{i}<1-\beta^i)=1-q^i.
\end{align*}
Altogether, using $G(0)=0$ and $G(u_{m_{n'}}(x_{n'}))=1$, \eqref{regcondprob}, \eqref{regcond2} and the above considerations yield
\begin{equation}\label{Pfinal}\begin{split}
\Pm(A_2\mid A_1)\leq & \quad G(u_{m_{n'}-j_{m_{n'}}(\beta)}(x_{n'}))\\
& +\sum_{i=\ell_{n'}+2}^{j_{m_{n'}}(\beta)}(1-q^i)\big(G(u_{m_{n'}-(i-1)}(x_{n'}))-G(u_{m_{n'}-i}(x_{n'}))\big)\\
& + (1-q^{\ell_{n'}+1})\big(1-G(u_{m_{n'}-\ell_{n'}+1}(x_{n'})\big).
\end{split}\end{equation}
Rewriting the terms in the sum as
\begin{align*}
& (1-q^i)\big(G(u_{m_{n'}-(i-1)}(x_{n'}))-G(u_{m_{n'}-i}(x_{n'}))\big)\\
& \quad =\big((1-q^{i-1})G(u_{m_{n'}-(i-1)}(x_{n'}))-(1-q^i)G(u_{m_{n'}-i}(x_{n'}))\big)+p\,q^{i-1}G(u_{m_{n'}-(i-1)}(x_{n'}))
\end{align*}
gives a telescopic expression plus some remainder and \eqref{Pfinal} can be simplified to
\begin{equation}\label{Pfinal1}\begin{split}
\Pm(A_2\mid A_1)\leq & \quad G(u_{m_{n'}-j_{m_{n'}}(\beta)}(x_{n'}))\\
& +(1-q^{\ell_{n'}+1})G(u_{m_{n'}-(\ell_{n'}+1)}(x_{n'})-(1-q^{j_{m_{n'}}(\beta)})G(u_{m_{n'}-j_{m_{n'}}(\beta)}(x_{n'}))\\
& +\sum_{i=\ell_{n'}+1}^{j_{m_{n'}}(\beta)-1} p\,q^i G(u_{m_{n'}-i}(x_{n'}))+ (1-q^{\ell_{n'}+1})\big(1-G(u_{m_{n'}-(\ell_{n'}+1)}(x_{n'})\big)\\
= & \quad 1-q^{\ell_{n'}+1}+q^{j_{m_{n'}}(\beta)}G(u_{m_{n'}-j_{m_{n'}}(\beta)}(x_{n'}))+\sum_{i=\ell_{n'}+1}^{j_{m_{n'}}(\beta)-1} p\,q^i G(u_{m_{n'}-i}(x_{n'})).
\end{split}\end{equation}
We now choose $\delta_{n'}:=u_{m_{n'}-j_{m_{n'}}(\beta)+\ell_{n'}+1}(x_{n'})=1+\beta^{m_{n'}-j_{m_{n'}}(\beta)+\ell_{n'}+1}x_{n'}\in(0,1)$ for sufficiently large $n'\in\nat$. Then $X_0\leq\delta_{n'}$ implies that $\Pm$-almost surely
$$X_i=\beta^i X_0+Y_i\leq\beta^i\delta_{n'}+1-\beta^i=1-\beta^i(1-\delta_{n'})$$
by \eqref{statsolrep} and \eqref{Yprop}. Note that $1-\beta^i(1-\delta_{n'})\leq u_{m_{n'}}(x_{n'})$ if and only if $\delta_{n'}-1\leq\beta^{m_{n'}-i}x_{n'}$ which is true for all $i\in\{1,\ldots,j_{m_{n'}}(\beta)-\ell_{n'}-1\}$ and sufficiently large $n'\in\nat$ such that $x_{n'}<0$. Hence for sufficiently large $n'\in\nat$ we have $\Pm$-almost surely
\begin{equation}\label{X0rel}
\{X_0\leq \delta_{n'},\,X_1\leq u_{m_{n'}}(x_{n'}),\ldots,X_{j_{m_{n'}}(\beta)-\ell_{n'}-1}\leq u_{m_{n'}}(x_{n'})\}=\{X_0\leq \delta_{n'}\}.
\end{equation}
Further note that by \eqref{statsolrep} we have
$$X_{j_{m_{n'}}(\beta)-\ell_{n'}-1}=\beta^{j_{m_{n'}}(\beta)-\ell_{n'}-1}X_{0}+Y_{j_{m_{n'}}(\beta)-\ell_{n'}-1}$$
and $Y_{j_{m_{n'}}(\beta)-\ell_{n'}-1}$ is independent of $X_{0}$. Hence we get
\begin{equation}\label{Pcompl}\begin{split}
& \Pm\left(X_{j_{m_{n'}}(\beta)-\ell_{n'}-1}>u_{m_{n'}-i}(x_{n'}), X_{0}\leq\delta_{n'}\right)\\
& \quad=\Pm\left(Y_{j_{m_{n'}}(\beta)-\ell_{n'}-1}>u_{m_{n'}-i}(x_{n'})-\beta^{j_{m_{n'}}(\beta)-\ell_{n'}-1}X_{0}, X_{0}\leq\delta_{n'}\right)\\
& \quad\geq\Pm\left(Y_{j_{m_{n'}}(\beta)-\ell_{n'}-1}>u_{m_{n'}-i}(x_{n'}), X_{0}\leq\delta_{n'}\right)\\
& \quad=\Pm\left(Y_{j_{m_{n'}}(\beta)-\ell_{n'}-1}>u_{m_{n'}-i}(x_{n'})\right)\cdot\Pm\left(X_{0}\leq\delta_{n'}\right).
\end{split}\end{equation}
Since $j_{m_{n'}}(\beta)>\ell_{n'}$ and $x_{n'}<0$ for sufficiently large $n'\in\nat$, we get 
$$u_{m_{n'}}(x_{n'})=1+\beta^{m_{n'}}x_{n'}\geq 1+\beta^{m_{n'}-j_{m_{n'}}(\beta)+\ell_{n'}+1}x_{n'}=\delta_{n'}$$
and it follows that
\begin{align*}
& \Pm\left(X_{j_{m_{n'}}(\beta)-\ell_{n'}-1}>u_{m_{n'}-i}(x_{n'}), X_{0}\leq\delta_{n'},X_{1}\leq u_{m_{n'}}(x_{n'}),\ldots,X_{j_{m_{n'}}(\beta)-\ell_{n'}-1}\leq u_{m_{n'}}(x_{n'})\right)\\
& \quad\leq\Pm\left(X_{j_{m_{n'}}(\beta)-\ell_{n'}-1}>u_{m_{n'}-i}(x_{n'}), X_{0}\leq u_{m_{n'}}(x_{n'}),\ldots,X_{j_{m_{n'}}(\beta)-\ell_{n'}-1}\leq u_{m_{n'}}(x_{n'})\right)\\
& \quad\leq\frac{\Pm\left(X_{j_{m_{n'}}(\beta)-\ell_{n'}-1}>u_{m_{n'}-i}(x_{n'}), X_{0}\leq u_{m_{n'}}(x_{n'}),\ldots,X_{j_{m_{n'}}(\beta)-\ell_{n'}-1}\leq u_{m_{n'}}(x_{n'})\right)}{\Pm\left(X_{0}\leq u_{m_{n'}}(x_{n'}),\ldots,X_{j_{m_{n'}}(\beta)-\ell_{n'}-1}\leq u_{m_{n'}}(x_{n'})\right)}\\
& \quad=\Pm\left(X_{j_{m_{n'}}(\beta)-\ell_{n'}-1}>u_{m_{n'}-i}(x_{n'})\,\Big|\,X_{0}\leq u_{m_{n'}}(x_{n'}),\ldots,X_{j_{m_{n'}}(\beta)-\ell_{n'}-1}\leq u_{m_{n'}}(x_{n'})\right)\\
& \quad=1-G\big(u_{m_{n'}-i}(x_{n'})\big).
\end{align*}
Hence for sufficiently large $n'\in\nat$ and $i\in\{1,\ldots,j_{m_{n'}}(\beta)-\ell_{n'}-1\}$ we get using \eqref{X0rel} and \eqref{Pcompl}
\begin{align*}
& G\big(u_{m_{n'}-i}(x_{n'})\big)\\
& \leq 1-\Pm\left(X_{j_{m_{n'}}(\beta)-\ell_{n'}-1}>u_{m_{n'}-i}(x_{n'}), X_{0}\leq\delta_{n'},X_{1}\leq u_{m_{n'}}(x_{n'}),\ldots,X_{j_{m_{n'}}(\beta)-\ell_{n'}-1}\leq u_{m_{n'}}(x_{n'})\right)\\
& =1-\Pm\left(X_{j_{m_{n'}}(\beta)-\ell_{n'}-1}>u_{m_{n'}-i}(x_{n'}), X_{0}\leq\delta_{n'}\right)\\
& \leq1-\Pm\left(Y_{j_{m_{n'}}(\beta)-\ell_{n'}-1}>u_{m_{n'}-i}(x_{n'})\right)\cdot\Pm\left(X_{0}\leq\delta_{n'}\right).
\end{align*}
Since by \eqref{statsolrep} and \eqref{distrfirst} we have
\begin{align*}
& \Pm\left(Y_{j_{m_{n'}}(\beta)-\ell_{n'}-1}>u_{m_{n'}-i}(x_{n'})\right)\\
& \quad=\Pm\left(X_{j_{m_{n'}}(\beta)-\ell_{n'}-1}>u_{m_{n'}-i}(x_{n'})-\beta^{j_{m_{n'}}(\beta)-\ell_{n'}-1}X_{0}\right)\\
& \quad\geq\Pm\left(X_{j_{m_{n'}}(\beta)-\ell_{n'}-1}>u_{m_{n'}-i}(x_{n'})\right)\\
& \quad=1-F_{\beta,p}(u_{m_{n'}-i}(x_{n'}))=q^{m_{n'}-i}\psi_{\beta,q}(x_{n'})
\end{align*}
it follows that
$$G\big(u_{m_{n'}-i}(x_{n'})\big)\leq 1-q^{m_{n'}-i}\psi_{\beta,q}(x_{n'})\cdot\Pm\left(X_{0}\leq\delta_{n'}\right)$$
and plugging this into \eqref{Pfinal1} yields
\begin{align*}
\Pm(A_2\mid A_1)\leq & \,1-q^{\ell_{n'}+1}+q^{j_{m_{n'}}(\beta)}\left(1-q^{m_{n'}-j_{m_{n'}}(\beta)}\psi_{\beta,q}(x_{n'})\cdot\Pm\left(X_{0}\leq\delta_{n'}\right)\right)\\
& +\sum_{i=\ell_{n'}+1}^{j_{m_{n'}}(\beta)-1} p\,q^i \left(1-q^{m_{n'}-i}\psi_{\beta,q}(x_{n'})\cdot\Pm\left(X_{0}\leq\delta_{n'}\right)\right)\\
= & \,1-q^{\ell_{n'}+1}+q^{j_{m_{n'}}(\beta)}-q^{m_{n'}}\psi_{\beta,q}(x_{n'})\cdot\Pm\left(X_{0}\leq\delta_{n'}\right)\\
& +p\sum_{i=\ell_{n'}+1}^{j_{m_{n'}}(\beta)-1} q^i - p\,q^{m_{n'}} \psi_{\beta,q}(x_{n'})\cdot\Pm\left(X_{0}\leq\delta_{n'}\right)\,\big(j_{m_{n'}}(\beta)-\ell_{n'}-1\big).
\end{align*}
Note that
$$p\sum_{i=\ell_{n'}+1}^{j_{m_{n'}}(\beta)-1} q^i =p\,q^{\ell_{n'}+1}\,\frac{1-q^{j_{m_{n'}}(\beta)-\ell_{n'}-1}}{1-q}=q^{\ell_{n'}+1}-q^{j_{m_{n'}}(\beta)}$$
and hence we get by Lemma \ref{maxconvmain}(a) and stationarity
\begin{align*}
\Pm(A_2\mid A_1) & \leq 1-\left(p\big(j_{m_{n'}}(\beta)-\ell_{n'}-1\big)+1\right)q^{m_{n'}}\psi_{\beta,q}(x_{n'})\cdot\Pm\left(X_{0}\leq\delta_{n'}\right)\\
& =\left(1-\Pm\left(X_{0}\leq\delta_{n'}\right)\right)+\Pm\left(X_{0}\leq\delta_{n'}\right)\left(1-\left(p\big(j_{m_{n'}}(\beta)-\ell_{n'}-1\big)+1\right)q^{m_{n'}}\psi_{\beta,q}(x_{n'})\right)\\
& =\left(1-\Pm\left(X_{0}\leq\delta_{n'}\right)\right)+\Pm\left(X_{0}\leq\delta_{n'}\right)\cdot\Pm\left(M_{j_{m_{n'}}(\beta)-\ell_{n'}}\leq u_{m_{n'}-i}(x_{n'})\right)\\
& =\left(1-\Pm\left(X_{0}\leq\delta_{n'}\right)\right)+\Pm\left(X_{0}\leq\delta_{n'}\right)\cdot\Pm(A_{2}).
\end{align*}
Since by \eqref{distrfirst} we have
$$\Pm\left(X_{0}\leq\delta_{n'}\right)=F_{\beta,p}(\delta_{n'}) = F_{\beta,p}(u_{m_{n'}-j_{m_{n'}}(\beta)+\ell_{n'}+1}(x_{n'}))=1-q^{m_{n'}-j_{m_{n'}}(\beta)+\ell_{n'}+1}\psi_{\beta,q}(x_{n'})$$
it follows again by Lemma \ref{maxconvmain}(a) and stationarity that
\begin{align*}
0\leq \Pm(A_2\mid A_1)-\Pm(A_{2}) & \leq \left(1-\Pm\left(X_{0}\leq\delta_{n'}\right)\right)+\Pm\left(X_{0}\leq\delta_{n'}\right)\cdot\Pm(A_{2})-\Pm(A_{2})\\
& =\left(1-\Pm\left(X_{0}\leq\delta_{n'}\right)\right)\cdot(1-\Pm(A_{2}))\\
& =\left(1-\Pm\left(X_{0}\leq\delta_{n'}\right)\right)\cdot \left(p\big(j_{m_{n'}}(\beta)-\ell_{n'}-1\big)+1\right)q^{m_{n'}}\psi_{\beta,q}(x_{n'})\\
& \leq C\left(j_{m_{n'}}(\beta)-\ell_{n'}\right)q^{2m_{n'}-j_{m_{n'}}(\beta)+\ell_{n'}}
\end{align*}
for some $C>0$, since $\psi_{\beta,q}(x_{n'})\to \psi_{\beta,q}(x)$. This shows \eqref{remain} and concludes the proof.
\end{proof}

\section{Connections with dynamical systems}

In \cite{BSS} it is stated that the reason why Bernoulli convolutions have drawn attention lies in its strong connection with deterministic dynamical systems discovered since the 1980's with reference to the origins \cite{AY}, \cite{PU}, and \cite{Led}. This connection is also given in section 4 of \cite{Sterk} and \cite{Bovier}, \cite{HolSte}, \cite{KemPer}. In this section we briefly describe further connections between the distribution function $F_{\beta,p}$ of the $\AR$ process \eqref{ar1} and both deterministic and random dynamical systems.

\subsection{Deterministic dynamics}

For $\beta, p \in (0,1)$ consider the following map:
\[
S_{\beta, p} : [0,1]^2 \to [0,1]^2, \quad
S_{\beta, p}(x,y) = 
\begin{cases}
	(\beta x, y/p)						& \text{if } y \in [0,p), \\
	(\beta x + (1-\beta), (y-p)/(1-p))	& \text{if } y \in [p,1].
\end{cases}
\]
For $p=\frac{1}{2}$ this map is topologically conjugate to the generalized baker's transformation introduced in \cite{AY}:
\[
T_\beta : [-1,1]^2 \to [-1,1]^2, \quad
T_\beta(x,y) = 
\begin{cases}
	(\beta x - (1-\beta), 2y+1) & \text{if } y \in [-1,0), \\
	(\beta x + (1-\beta), 2y-1) & \text{if } y \in [0,1].
\end{cases}
\]
Indeed, a straightforward computation shows that $h^{-1} \circ T_\beta \circ h = S_{\beta,1/2}$ where the conjugacy $h : [0,1]^2 \to [-1,1]^2$ is given by $h(x,y)=(2x-1,2y-1)$. Therefore, the maps $S_{\beta,1/2}$ and $T_\beta$ have the same dynamics. We can interpret $S_{\beta,p}$ as a further generalization of $T_\beta$, albeit in a different coordinate system.

We briefly sketch how the distribution function $F_{\beta,p}$ yields an $S_{\beta,p}$-invariant measure on the state space $[0,1]^2$. The sets $(a,b] \times (c,d] \subseteq [0,1]^2$ form a semi-algebra. The map
\[
\mu_{\beta,p}((a,b] \times (c,d]) = (F_{\beta,p}(b)-F_{\beta,p}(a))(d-c)
\]
can be shown to be countably additive and hence extends to a measure on the generated Borel $\sigma$-algebra.

\begin{lemma}
For all Borel-measurable sets $A \subseteq [0,1]^2$ we have
\[
\mu_{\beta,p}(S_{\beta,p}^{-1}(A)) = \mu_{\beta,p}(A).
\]
\end{lemma}

\begin{proof}
It suffices to check this for the generators of the $\sigma$-algebra. For $(a,b] \times (c,d] \subseteq [0,1]^2$ we have the following disjoint union:
\[
S_{\beta,p}^{-1}( (a,b] \times (c,d] ) = A_1 \cup A_2,
\]
where
\[
\begin{split}
A_1 & = (a/\beta, b/\beta] \times (pc, pd], \\
A_2 & = (a/\beta+1-1/\beta,b/\beta+1-1/\beta] \times (p+qc, p+qd]
\end{split}
\]
are disjoint sets due to $p+qc\geq p\geq pd$. The definition of $\mu_{\beta,p}$ gives
\[
\begin{split}
	\mu_{\beta,p}(A_1) & = (F_{\beta,p}(b/\beta) - F_{\beta,p}(a/\beta))(pd-pc), \\
	\mu_{\beta,p}(A_2) & = (F_{\beta,p}(b/\beta+1-1/\beta) - F_{\beta,p}(a/\beta+1-1/\beta))(qd-qc).
\end{split}
\]
Adding the results and using the functional equation \eqref{dfstatsol} gives
\[
\mu_{\beta,p}(S_{\beta,p}^{-1}( (a,b] \times (c,d] ))
	= (F_{\beta,p}(b) - F_{\beta,p}(a))(d-c)
	= \mu_{\beta,p}((a,b]\times (c,d]),
\]
as desired.
\end{proof}

In the particular case $p=\frac{1}{2}$ we obtain (modulo the coordinate transformation given by the conjugacy $h$) the invariant measure that was introduced in \cite{AY} for the map $T_\beta$; it was proven that this measure is of Bowen-Ruelle type and strongly mixing.

The map $S_{\beta,p}$ is a skew-product: the dynamics in the $y$-direction drives the dynamics in the $x$-direction but not vice versa. The driving map
\[
g : [0,1] \to [0,1], \quad
g(y) =
\begin{cases}
	y/p				& \text{if } y \in [0,p), \\
	(y-p) / (1-p)	& \text{if } y \in [p,1],
\end{cases}
\]
has the Lebesgue measure as an ergodic invariant measure. Birkhoff's Ergodic Theorem implies that the iterates of $g$ land in the intervals $[0,p)$ and $[p,1]$ with probability $p$ and $q$, respectively. The iterates of $g$ determine which of the maps $x \mapsto \beta x$ and $x \mapsto \beta x + (1-\beta)$ is applied. In addition, the map $g$ has exponential decay of correlations \cite{BG}. In this way, the projection of the iterations of $S_{\beta,p}$ onto the $x$-axis can be interpreted as a deterministic approximation of the $\AR$ process \eqref{ar1}. This perspective connects with recent research on establishing extreme value laws for time series obtained by evaluating a scalar observable along orbits of a deterministic dynamical system; see \cite{LFFFHKNTV} and references therein.

\subsection{Random dynamics}

The distribution function $F_{\beta,p}$ also appears when studying diverging orbits obtained from randomly iterating two functions $f_0, f_1 : \mathbb{R} \to \mathbb{R}$ that satisfy appropriate conditions.
The set of all binary sequences $\Omega = \{(\omega_n)_{n \in \mathbb{N}} : \omega_n \in \{0,1\}\}$ can be equipped with probability measures as follows. First, for cylinder sets
\[
[\alpha_1,\dots,\alpha_n] = \{ \omega \in \Omega : \omega_i = \alpha_i \text{ for all } i=1,\dots,n \}
\]
with $\alpha_{i}\in\{0,1\}$ we define
\[
\mu_p([\alpha_1,\dots,\alpha_n]) = p^{n-(\alpha_1 + \dots + \alpha_n)} q^{\alpha_1 + \dots + \alpha_n}.
\]
Next, $\mu_p$ can be extended to a unique probability measure, known as the Bernoulli measure, on the $\sigma$-algebra generated by the collection of all cylinder sets.

For an initial point $x_0 \in \mathbb{R}$ and a symbol sequence $\omega \in \Omega$ we define the following sequence of iterates:
\[
x_n = \bigcirc_{i=1}^n f_{\omega_i}(x_0) = (f_{\omega_n} \circ \dots \circ f_{\omega_2} \circ f_{\omega_1})(x_0), \quad n \in \mathbb{N}.
\]
In \cite{MS} the following monotonicity assumptions were made:
\begin{itemize}
\item[(A1)] The functions $f_0$ and $f_1$ are continuous and strictly increasing.
\item[(A2)] There exist real numbers $0 < a \leq b < 1$ such that
\[
f_0(0) = 0, \quad
f_0(a) = 1, \quad
f_1(b) = 0, \quad
f_1(1) = 1.
\]
\item[(A3)] There exists a real number $s > 1$ such that the following implications hold:
\[
\begin{split}
x \leq 0 & \quad\Rightarrow\quad f_0(x) \leq s x, \\
x \geq a & \quad\Rightarrow\quad f_0(x) \geq 1+s(x-a), \\
x \leq b & \quad\Rightarrow\quad f_1(x) \leq s(x-b), \\ 
x \geq 1 & \quad\Rightarrow\quad f_1(x) \geq 1+s(x-1).
\end{split}
\]
These inequalities ensure that iterates landing outside the unit interval will diverge to either plus or minus infinity.
\item[(A4)] Optionally, the following symmetry condition can be imposed: for all $x \in \mathbb{R}$ we have
\[
f_0(1-x) + f_1(x) = 1.
\]
\end{itemize}

In \cite{MS} it was shown that the function
\[
G_{p} : \mathbb{R} \to \mathbb{R}, \quad
G_{p}(x) = \mu_p(\{ \omega \in \Omega \,:\, \bigcirc_{i=1}^n f_{\omega_i}(x) \to \infty \})
\]
is continuous and satisfies the following properties under assumptions (A1)--(A3):
\begin{enumerate}
\item[(i)] $G_{p}(x) = 0$ for all $x \leq 0$ and $G_{p}(x) = 1$ for all $x \geq 1$;
\item[(ii)] $G_{p}(x) \leq G_p(y)$ whenever $x \leq y$;
\item[(iii)] $G_p(x) = p\,G_p(f_0(x)) + q\,G_p(f_1(x))$ for all $x \in \mathbb{R}$;
\item[(iv)] $G_{p}(1-x) = 1-G_{q}(x)$ for all $x \in \mathbb{R}$ whenever assumption (A4) holds in addition.
\end{enumerate}
For the particular choice $f_0(x) = x/\beta$ and $f_1(x) = x/\beta + 1-1/\beta$ with $\beta \in (0,\frac{1}{2}]$ the assumptions (A1)--(A4) are fulfilled and the function $G_p$ coincides with $F_{\beta,p}$ since both functions are continuous distribution functions satisfying the functional equation \eqref{dfstatsol}; see Lemma \ref{uniqueness}.


\begin{thebibliography}{99}

\bibitem{AY} Alexander, J.C.; and Yorke, J.A. (1984) Fat baker's transformations. {\it Ergodic Theory Dynam. Systems} {\bf 4} 1--23.

\bibitem{BSS} B\'ar\'any, B.; Simon, K.; and Solomyak, B. (2023) {\it Self-similar and Self-affine Sets and Measures.} Amer. Math. Soc., Providence.

\bibitem{BDGPS} Bertin, M.J.; Decomps-Guilloux, A.; Grandet-Hugot, M.; Pathiaux-Delefosse, M.; and Shreiber, J.P. (1992) {\it Pisot and Salem Numbers.} Birkh\"auser, Basel.

\bibitem{Bovier} Bovier, A. (1996) Bernoulli convolutions, dynamical systems and automata. {\it Disordered Systems (Temuco 1991/1992). Travaux en Cours} {\bf 53}, Herman, Paris, pp. 63--86.

\bibitem{BG} Boyarski, A.; and G\'ora, P. (1997) {\it Laws of Chaos: Invariant Measures and Dynamical Systems in One Dimension.} Birkh\"auser, Boston.

\bibitem{BroDav} Brockwell, P.J.; and Davis, R.A. (2006) {\it Time Series: Theory and Methods.} Second Edition, Springer, New York.

\bibitem{BDM} Buraczewski, D.; Damek, E.; and Mikosch, T. (2016) {\it Stochastic Models with Power-Law Tails -- The Equation $X=AX+B$.} Springer, Cham.

\bibitem{CHT} Canto e Castro, L.; de Haan, L,; and Temido, M.G. (2001) Rarely observed sample maxima. {\it Theory Probab. Appl.} {\bf 45} 658--661.

\bibitem{Cher} Chernick, M.R. (1981) A limit theorem for the maximum of autoregressive processes with uniform marginal distribution. {\it Ann. Probab.} {\bf 9} 145--149.

\bibitem{DiaFre} Diaconis, P.; and Freedman, D. (1999) Iterated random functions. {\it SIAM Review} {\bf 41}(1) 45--76.

\bibitem{Erdoes} Erd\H{o}s, P. (1939) On a family of symmetric Bernoulli convolutions. {\it Amer. J. Math.} {\bf 61} 974-975.

\bibitem{Gal} Galambos, J. (1978) {\it The Asymptotic Theory of Extreme Order Statistics.} Wiley, New York.

\bibitem{Gar} Garsia, A.M. (1962) Arithmetic properties of Bernoulli convolutions. {\it Trans. Amer. Math. Soc.} {\bf 102} 409--432.

\bibitem{GMS} Glava\v{s}, L.; Mladenovi\'c, P.; and Samorodnitsky, G. (2017) Extreme values of the uniform order 1 autoregressive processes and missing observations. {\it Extremes} {\bf 20} 671--690.

\bibitem{GolMal} Goldie, C.M.; and Maller, R.A. (2000) Stability of perpetuities with thin tails. {\it Ann. Probab} {\bf 28} 1196--1218.

\bibitem{Gri1} Grinevich, I.V. (1992) Max-semistable limit distributions corresponding to linear and power normalizations. {\it Theory Probab. Appl.} {\bf 37} 720--721.

\bibitem{Gri2} Grinevich, I.V. (1993) Domains of attraction of the max-semistable laws under linear and power normalizations. {\it Theory Probab. Appl.} {\bf 38} 640--650.

\bibitem{HF} de Haan, L.; and Ferreira, A. (2006) {\it Extreme Value Theory.} Springer, New York.

\bibitem{HolSte} Holland, M.P.; and Sterk, A.E. (2021) On max-semistable laws and extremes for dynamical systems. {\it Entropy} {\bf 23} 1192.

\bibitem{JesWin} Jessen, B.; and Wintner, A. (1935) Distribution functions and the Riemann zeta function. {\it Trans. Amer. Math. Soc.} {\bf 38} 48--88.

\bibitem{KemPer} Kempton, T.; and Persson, T. (2015) Bernoulli convolutions and 1D dynamics. {\it Nonlinearity} {\bf 28} 3921--3934.

\bibitem{KerWin} Kershner, R.; and Wintner. A. (1935) On symmetric Bernoulli convolutions. {\it Amer. J. Math.} {\bf 57} 541--548.

\bibitem{LLR} Leadbetter, M.R.; Lindgren, G; and Rootz\'en, H. (1983) {\it Extremes and Related Properties of Random Sequences and Processes.} Springer, New York.

\bibitem{Led} Ledrappier, F. (1992) On the dimension of some graphs. In: Symbolic dynamics and its applications (New Haven 1991). {\it Contemp. Math.} {\bf 135}. Amer Math. Soc., Providence, pp. 285--293.

\bibitem{LFFFHKNTV} Lucarini, V.; Faranda, D.; Freitas, A.C.M.; Freitas, J.M.; Holland, M.; Kuna, T.; Nicol, M.; Todd, M.; and Vaienti, S. (2016) {\it Extremes and Recurrence in Dynamical Systems.} Wiley, Hoboken, New Jersey.

\bibitem{MS} Mitrea, C.; and Sterk, A.E. (2025) Singular functions obtained via random function iteration. Accepted for publication in {\it Involve, a Journal of Mathematics}.

\bibitem{Megyesi} Megyesi, Z. (2002) Domains of geometric partial attraction of max-semistable laws: Structure, merge and almost sure limit theorems. {\it J. Theoret. Probab.} {\bf 15}(4) 973--1005.

\bibitem{MFS} Moloney, N.R.; Faranda, D.; and Sato, Y. (2019) An overview of the extremal index. {\it Chaos} {\bf 29} 022101.

\bibitem{Pan} Pancheva, E. (2010) Max-semistability: a survey. {\it ProbStat Forum} {\bf 3} 11-24.

\bibitem{PerSch} Peres, Y.; and Schlag, W. (2000) Smoothness of projections, Bernoulli convolutions, and the dimension of exceptions. {\it Duke Math. J.} {\bf 102}(2) 193--251.

\bibitem{PSS} Peres, Y.; Schlag, W.; and Solomyak, B. (2000) Sixty years of Bernoulli convolutions. In: Bandt, C. et al. (Eds.) {\it Fractal Geometry and Stochastics II.} Progress in Probability {\bf 46}, Birkh\"auser, Basel, pp. 39--65.

\bibitem{PerSol1} Peres, Y.; and Solomyak, B. (1996) Absolute continuity of Bernoulli convolutions, a simple proof. {\it Math. Res. Lett.} {\bf 3} (2) 231-239.

 \bibitem{PerSol2} Peres, Y.; and Solomyak, B. (1998) Self-similar measures and intersections of Cantor sets. {\it Trans. Amer. Math. Soc.} {\bf 350}(10) 4065-4087.

\bibitem{PU} Przytycki, F.; and Urba\'nski, M. (1989) On the Hausdorff dimension of some fractal sets. {\it Studia Math.} {\bf 93} 155-186.

\bibitem{Res} Resnick, S.I. (1987) {\it Extreme Values, Regular Variation, and Point Processes.} Springer, New York.

\bibitem{Sol1} Solomyak, B. (1995) On the random series $\sum\pm\lambda^{i}$ (an Erd\H{o}s problem). {\it Ann. Math.} {\bf 142} 611--625.

\bibitem{Sol2} Solomyak, B. (2004) Notes on Bernoulli convolutions. In: Lapidus, M.L. et al. (Eds.) {\it Fractal Geometry and Applications: A Jubilee of Beno\^{i}t Mandelbrot.} Proceedings of Symposia in Pure Mathematics {\bf 72} Part 1, AMS, Providence, pp. 207--232.

\bibitem{Sterk} Sterk, A.E. (2025) Max-semistable extreme value laws for autoregressive processes with Cantor-like marginals. {\it Extremes} {\bf 28} 371--391.

\bibitem{TemCan} Temido, M.G.; and Canto e Castro, L. (2002) Max-semistable laws in extremes of stationary random sequences. {\it Theory Probab. Appl.} {\bf 47} 365--374.

\bibitem{Varju} Varj\'u, P.P. (2016) {\it Recent Progress on Bernoulli Convolutions.} Apollo - University of Cambridge Repository. {\tt https://doi.org/10.17863/CAM.17352}

\bibitem{Varju2} Varj\'u, P.P. (2019) Absolute continuity of Bernoulli convolutions for algebraic parameters. {\it J. Amer. Math. Soc.} {\bf 32} 351--397.

\bibitem{Win} Wintner, A. (1935) On convergent Poisson convolutions. {\it Amer. J. Math.} {\bf 57} 827--838.

\bibitem{Yu} Yu, H. (2022) Bernoulli convolutions with Garsia parameters in $(1,\sqrt{2}]$ have continuous density functions. {\it Proc. Amer. Math. Soc.} {\bf 150} 4359--4368.

\end{thebibliography}
\end{document}